\documentclass[12pt, leqno]{article}

\usepackage{latexsym}
\usepackage{indentfirst}
\usepackage{graphicx}
\usepackage{amsmath,amsfonts}
\usepackage{amssymb}
\usepackage{amsthm}
\usepackage{bm}
\usepackage{accents}
\usepackage{changepage}
\usepackage{color}
\usepackage{bigints}
\usepackage{enumerate}
\usepackage{amsmath, nccmath}
\usepackage[margin=1in]{geometry}
\usepackage{yhmath}
\usepackage{dirtytalk}
\usepackage{hyperref}
\usepackage{cleveref}
\usepackage{comment}

\newtheorem{prop}{Proposition}[section]
\newtheorem*{defi*}{Definition}
\newtheorem{defi}[prop]{Definition}

\newtheorem{lem}[prop]{Lemma}
\newtheorem{rem}[prop]{Remark}
\newtheorem{thm}[prop]{Theorem}
\newtheorem{coro}[prop]{Corollary}

\newcommand{\cb}{{\mathcal B}}

\newcommand{\ce}{{\mathcal E}}

\newcommand{\cp}{{\mathcal P}}
\newcommand{\calr}{{\mathcal R}}  
   
\newcommand{\ct}{{\mathcal T}}
\newcommand{\cu}{{\mathcal U}}   
    
\newcommand{\cv}{{\mathcal V}}
\newcommand{\cw}{{\mathcal W}}

\newcommand{\R}{\mathbb{R}}

\newcommand{\al}{{\alpha}}

\newcommand{\ex}{\mbox{\sf Exc}_{\cal U}}

\numberwithin{equation}{section}

\begin{document}

\noindent
{\Large\bf Path continuity of Markov processes and locality of\\[1mm] 
Kolmogorov operators
}\\

\noindent
Lucian Beznea\footnote{Simion Stoilow Institute of Mathematics  of the Romanian Academy,
 Research unit No. 2, 
P.O. Box \mbox{1-764,} RO-014700 Bucharest, Romania, and  University of Bucharest, 
Faculty of Mathematics and Computer Science (e-mail: lucian.beznea@imar.ro)},
Iulian C\^{i}mpean\footnote{ University of Bucharest, Faculty of Mathematics and Computer Science, and Simion Stoilow Institute of Mathematics of the Romanian Academy, Research unit No. 2, P.O. Box \mbox{1-764,} RO-014700 Bucharest Romania (e-mail: iulian.cimpean@unibuc.ro; iulian.cimpean@imar.ro)}
and 
Michael R\"ockner\footnote{Fakult\"at f\"ur Mathematik, Universit\"at Bielefeld,
Postfach 100 131, D-33501 Bielefeld, Germany, and Academy for Mathematics and Systems Science, CAS, Beijing 
(e-mail: roeckner@mathematik.uni-bielefeld.de)}
\\[5mm]

\paragraph{Abstract.} We prove that if we are given a generator of a c\` adl\` ag Markov process %$X$ 
and an open domain $G$ in the state space, on which the generator has the local property expressed in a suitable way on a class $\mathcal{C}$ of test functions that is sufficiently rich,  
%(yet not necessarily a core), 
then the Markov process  has continuous paths when it passes through $G$. 
%Because we do not need to work with cores, 
The result holds for any Markov process which is associated with the generator merely on $\mathcal{C}$. 
This points out that the path continuity of the process is an a priori property encrypted by the generator acting on enough test functions, and this property can be easily checked in many situations. 
The approach uses potential theoretic tools and covers Markov processes associated with (possibly time-dependent) second order integro-differential operators (e.g., through the martingale problem)  defined on domains in Hilbert spaces or on spaces of measures.
\\[2mm]

\\

\noindent
{\bf Mathematics Subject Classification (2010):} 
%60H15,     % Stochastic partial differential equations [See also 35R60]
%60H10,     % Stochastic ordinary differential equations [See also 34F05]
60J45,  	% Probabilistic potential theory
60J35,    	% Transition functions, generators and resolvents
60J40,  	% Right processes
60J57,  	% Multiplicative functionals
31C25,      % Dirichlet spaces
47D07,  	% Markov semigroups and applications to diffusion processes 
60J25,      % Continuous-time Markov processes on general state spaces
60J60.      % Diffusion processes

%37C40 (primary), 37A30, 37L40, 60J35, 60J25, 60J60, 31C25, 37C40,  82B10 (secondary)
%47D03  	%Groups and semigroups of linear operators 
%47A35  	%Ergodic theory
%37A30   %Ergodic theorems, spectral theory, Markov operators
%37C40   %Smooth ergodic theory, invariant measures
%37L40   %Invariant measures (Infinite-dimensional dissipative dynamical systems)
%60J55  	Local time and additive functionals
%60J25  	Continuous-time Markov processes on general state spaces
%60J35: Transition functions, generators and resolvents
%82B10  	Quantum equilibrium statistical mechanics (general)
%31C05  	Harmonic, subharmonic, superharmonic functions
% 60J45: Probabilistic potential theory
% 60J60: Diffusion processes
%35R60,       % Partial differential equations with randomness, stochastic partial differential equations
\tableofcontents

\section{Introduction}

Since Fick and Einstein, diffusion has been for
more than a century and a half a key concept in many fields, from physics to biology or finance. 
It is used to describe the transition of anything like atoms, molecules, particles etc, from a region of higher concentration to a region of lower concentration, in a continuous way; see the “celebration” paper \cite{Ph05}. 
Formally, in this paper, by a diffusion process we understand a Markov process with almost surely
continuous trajectories. 

In general, diffusion Markov processes correspond to second order elliptic
differential operators (see \cite{Fe54} or the monograph \cite{Ba98}). 
However, starting from a given differential operator or a semigroup of Markov operators, it is a highly non-trivial task to rigorously show that one can construct an associated Markov process with continuous paths. 
A seminal contribution to this subject was given by D. Ray in \cite{Ra56}, proving a conjecture of W. Feller stated in \cite{Fe54} which
asserts that on the real line, if $(P_t)_{t\geq 0}$ is a Feller transition semigroup of a (c\'adl\'ag) strong Markov
process, then the Lindeberg-type condition $P_t(x,\mathbb{R}\setminus(x-\varepsilon,x+\varepsilon))=o(t)$ uniformly in $x$ on compact sets for each $\varepsilon>0$ ensures the a.s. continuity of the trajectories (see also \cite{He81}). 
This result has also been extended to locally compact spaces, as e.g. in \cite[Proposition 9.10]{BlGe68}; see also \Cref{prop:Getoor} below for the precise statement.
A characterisation of the Lindeberg-type condition in terms of the associated "superharmonic function" (the local truncation property)
was given in \cite{BlHa86}, Proposition 8.2, in the frame of the balayage spaces.
Anyway, there are many situations where such a result can not be
applied, e.g. when the underlying space in not locally compact or the semigroup is not Feller. 
Another typical and general situation where more sophisticated tools have to be used in order to construct an associated Markov process with
continuous paths occurs when we start from the Kolmogorov operator defined merely on a class of test functions, e.g. when the latter is associated with a stochastic (partial)
differential equations with singular coefficients (e.g. like those considered in \cite{DaRo02} or \cite{BeCiRo20}). 
When a convenient duality measure exits, the theory of Dirichlet form offers powerful tools in this respect.
More precisely, it is well known that in the case of (non-symmetric) Dirichlet forms, the associated Markov process is a diffusion if and only if the form is local, i.e. the relative energy of any two elements from the energy space with disjoint compact supports is zero, see e.g. the monographs \cite{FuOsTa11} or \cite{MaRo92}.
This result was further extended to generalized (non-sectorial) Dirichlet forms in \cite[Proposition 1.10]{St99a}, \cite[Theorem 3.3]{Tr03}, and \cite[Proposition 3.6]{LeStTr22} under suitable conditions. 
At the core of the Dirichlet forms approach stand the orthogonal decomposition of elements from the domain
of the form using the hitting distribution (the balayage operator, in potential theoretic terms), as well as the strong duality theory existing in such energetic spaces.

Having in mind the above mentioned results, it seems somehow frustrating that, on the one hand, just by looking at a second order operator defined on a class of test functions one can easily guess if an associated c\`adl\`ag process should have continuous paths, and on the other hand, the above established results in this direction require a lot of structure and regularity for the semigroup or the generator to ensure what at an intuitive level might seem obvious.
As a matter of fact, following another approach, namely the one developed by D. Bakry and M. \'Emery in \cite{BaEm85}, it is indeed possible to analyse the path continuity of a Markov process in a more direct way, as soon as it corresponds to an operator (through the martingale problem) defined on a class of test functions whose square field operator satisfies the so called {\it derivation property}.
However, the existence of the square field operator requires a specific algebraic structure. 
In particular, the class of test functions needs to be an algebra which is invariant under the action of the operator, and this usually leads to restrictions on the coefficients.

The previous discussion raises the following fundamental problem which we address in this paper: Suppose that we are given a second order operator ${\sf L}$ defined merely on a class of test functions $\mathcal{D}_0$, which is associated to a c\`adl\`ag Markov process $X$, on a general topological space, with transition function $(P_t)_{t\geqslant 0}$ e.g.,  
through the corresponding martingale problem.
Further, suppose that $({\sf L},\mathcal{D}_0)$ exhibits locally a local character in the informal sense that ${\sf L}u$ vanishes where $u$ vanishes for $u\in \mathcal{D}_0$, but merely on a fixed open set $G\subset E$, hence $\sf L$ could also be an integro-differential operator whose non-local part acts only outside $G$. 
Given $({\sf L}, \mathcal{D}_0,X,G)$ as above, can we decide that $X$ has continuous paths when it lies in $G$?
Note that in the above context it is not assumed that $\mathcal{D}_0$ is a core or that $L$ is associated to a nice Dirichlet space. 
Also, $\sf L$ is allowed to have irregular coefficients, so that $\mathcal{D}_0$ is not necessarily invariant under the action of $\sf L$. 
In a nutshell, we show that if $\mathcal{D}_0$ is rich enough (yet not necessarily a core) and $(P_t)_{t\geqslant 0}$ has some minimal regularity so that $X$ is at least strong Markov, then the answer is affirmative.
To this end, we first adopt a general $L^p$-approach (see the main result \Cref{thm 2.2} below), and then we show that, in fact, this approach leads to a similar result \Cref{coro 2.6} which does not require an $L^p$-context, but only the process, its transition function, and a choice of the underlying topology.
We show that the obtained theoretical results are applicable to large classes of examples that cover Markov processes associated (e.g. through the martingale problem) to (possibly time-dependent) second order integro-differential operators defined on domains in Hilbert spaces or on spaces of measures.

The structure of the paper is the following: In \Cref{s:preliminaries} we introduce the context and present some preliminary results, most of them being of self interest. 
More precisely, we discuss the problem of existence of a regular copy of a given simple c\`adl\`ag Markov process, the concept of diffusion and related potential theoretic tools, and some considerations on infinitesimal generators for resolvents on $L^p$-spaces.
\Cref{s:main} is devoted to the main theoretical results, namely \Cref{thm 2.2}, \Cref{coro 2.6}, \Cref{coro:nonh}, and \Cref{coro:D_m}.
Then, in \Cref{s:applications} we apply the theoretical results to a large class of examples: Jump-Diffusions on domains in Hilbert spaces, measure-valued branching processes, diffusions associated with generalized Dirichlet forms, and even diffusions associated with not necessarily quasi-regular semigroups.
Finally, we also included an Appendix aimed to rapidly introduce the reader to the main potential theoretical notions and results that are employed in the main body of this paper.

We would like to end the introduction by pointing out a subtle aspect: In contrast to the framework of Dirichlet forms which are first order objects and cover operators in both divergence and non-divergence form, the present analysis is designed mainly for operators ${\sf L}$ in non-divergence form. 
In favour of the large class of herein considered applications, this is a small price we accepted to pay. 
Nevertheless, our results definitely apply also to generalized Dirichlet forms rendering simple conditions to check the path continuity of the associated process, as done in \Cref{coro:GDF}

\section{The framework and preliminary results}\label{s:preliminaries}

To achieve our main results from \Cref{s:main}, we need several results of general nature, which in our opinion are of self-interest. 
More precisely, let us address in the sequel three preliminary topics: \;1) \textit{the problem of existence of a regular copy of a given simple c\`adl\`ag Markov process;} \; 2) \textit{the concept of diffusion and related potential theoretic tools;} \; 3) \textit{resolvents of kernels and some considerations on their infinitesimal generators on $L^p$-spaces.} 

\paragraph{From Markov processes with c\`adl\`ag paths to right processes.} 
Let $(E,\tau)$ be a Lusin topological spaces whose Borel $\sigma$-algebra is denoted by $\mathcal{B}:=\mathcal{B}(E)$.
For each $x\in E$, let $(X_t^x)_{t\geq 0}$ be a c\`adl\`ag temporally homogeneous Markov process defined on $(\Omega, \mathcal{F},\mathcal{F}_t,\mathbb{P}^{x})$, with transition function $(P_t)_{t\geq 0}$ and initial distribution $\delta_x$; just for the sake of generality, here $\Omega, \mathcal{F} \mbox{ or } \mathcal{F}_t$ may as well depend on $x$. 
We denote by $\mathcal{U}=(U_\alpha)_{\alpha>0}$ the corresponding resolvent family of Markov kernels, namely
\begin{equation*}
    U_\alpha f(x) = \int_0^\infty e^{-\alpha t} P_tf(x)\;dt \quad \mbox{for all }f\in b\mathcal{B}, x\in E \mbox{ and } \alpha>0.
\end{equation*}
As announced, the main goal of this paper is to study when $X:=((X^x_t)_{t\geq 0}, x\in E)$ has continuous paths. 
However, a preliminary question we wish to address here is whether one can construct (and hence work with) a more regular copy of $X$, namely a right process sharing the same finite dimensional distributions, or equivalently, having resolvent $\mathcal{U}$. 
If this can be done, the probabilistic potential theory would then be in force. 
In particular, the measurability properties of hitting times, excessive functions, and hitting distributions (or balayage operators) would be guaranteed, as well as the strong Markov property.
To do so, we consider the following hypothesis, which is quite natural to impose in order to guarantee the strong Markov property, especially if one reads it as {\it generalized Feller property}; in fact, it is essentially the one from \cite{BeRo11a}, page 846.

\vspace{0.2 cm}
\noindent
{$\mathbf{(H_0)}$}  There exists a vector lattice $\mathcal{C}\subset C_b(E)$ such that
\begin{enumerate}[(i)]
\item $1 \in \mathcal{C}$ and there exists a countable subset in $\mathcal{C}$ which separates the points of $E$.
\item We have ${U}_\beta f \in \mathcal{C}$  for all $f \in \mathcal{C}$ and $\beta >0$.
\end{enumerate}

\begin{prop} \label{prop2.1}
If $X$ and $\mathcal{U}$ are as above and $\mathbf{(H_0)}$ is satisfied then there exists a right process on $X'$ on $E$ which is c\`adl\`ag with respect to $\tau$, sharing the same resolvent $\mathcal{U}$. 
In particular, $X$ and $X'$ have the same law on the Skorokhod space of all c\`adl\`ag paths from $[0,\infty)$ to $E$.

\medskip
\noindent{Proof in \Cref{pf:1}.}
\end{prop}

\begin{rem}
\begin{enumerate}[(i)]
\item If $\mathcal{U}$ is (Lipschitz) Feller, i.e. for some (hence all) $\alpha >0$ it holds that $U_\alpha$ maps bounded (Lipschitz) continuous functions to (Lipschitz) continuous functions on $E$, then $\mathbf{(H_0)}$ is clearly satisfied. 
\item If $\mathbf{(H_0)}$ is fulfilled, a major benefit of Proposition \ref{prop2.1} is that the path continuity properties for a simple c\`adl\`ag Markov process can be in fact studied for a more regular version of it, for which potential theoretical tools are available. 
\item It is possible to ensure directly the existence of a c\'adl\'ag right process given a transition function, by means of potential theoretic tools. 
For the reader convenience we included one such result in Appendix, \Cref{thm:cadlag}.
\end{enumerate}
\end{rem}

%%%%%%%%%%%%%%%%%%  end paragraph

\paragraph{Right processes and diffusions.}
From now on, throughout this section, we assume that $X=(\Omega, \mathcal{F}, \mathcal{F}_t, X_t, \theta_t, \mathbb{P}^x)$ is a right Markov process on a Lusin topological space $(E,\tau)$ with Borel $\sigma$-algebra $\mathcal{B}$, which is  c\`adl\`ag  w.r.t. $\tau$; in particular, $\tau$ is a natural topology on $E$. 
The lifetime and cemetry point of the process are denoted by $\zeta$ and $\Delta$, respectively; we extend any function $u:E\rightarrow \mathbb{R}$ to $E\cup \{\Delta\}$ by setting $u(\Delta)=0$.
The resolvent of $X$ is denoted by $\mathcal{U}=(U_\alpha)_{\alpha>0}$; if $\beta>0$, we set $\mathcal{U}_\beta:=(U_{\alpha+\beta})_{\alpha>0}$.
The first hitting time of a set $A \in \mathcal{B}$ by the process $X$ is defined by 
$$T_A := \inf \{ t > 0 : X_t \in A \}.$$

Recall that $m$ is called a {\it reference} measure if $m(A)=0$ implies that $A$ is $\mathcal{U}$-negligible, i.e. $U_1(1_A)\equiv 0$.
It is easy to check that if $\mathcal{U}$ is {\it strong Feller} (i.e. $U_\alpha (b\mathcal{B})\subset C_b(E)$ for one (hence all) $\alpha>0$) and $supp(m)=E$, then $m$ is a reference measure.

The following result is a main tool to show that if a process has continuous paths except some $m$-negligible set, then it has automatically continuous paths except some $m$-innesential set.

\begin{prop} \label{prop 2.4} 
The following assertions hold.
\begin{enumerate}[(i)]
\item The function $v(x):= \mathbb{P}^{x} (\{\omega : [0, \zeta(\omega)) \ni t \longmapsto X_t(\omega) \mbox{ is not continuous} \} )$, $x\in E$, is  $\mathcal{U}$-excessive.
\item If
$\mathbb{P}^{x}(\{\omega : [0, \zeta(\omega)) \ni t\mapsto X_t(\omega) \mbox{ is continuous}\})=1 \;\; \mathcal{U}\mbox{-a.e.},
$
then the equality holds for all $x\in E$.

In particular, if $m$ is a reference measure and the above equality holds $m$-a.e., then it holds for all $x\in E$.
\item If $m$ is a $\sigma$-finite measure on $E$ and
$$
\mathbb{P}^{x}(\{\omega : [0, \zeta(\omega)) \ni t \longmapsto X_t(\omega) \mbox{ is continuous}\})=1 \;\; m\mbox{-a.e.},
$$
then the equality holds $m$-q.e. 
\end{enumerate}

\medskip
\noindent{Proof in \Cref{pf:1}.}
\end{prop}

\begin{defi} \label{defi:IG}
Let $G \subset E$ be an open set.
For a sequence $G_n \subset G, n \geq 1$ of open subsets such that $\overline{G}_n\subset G_{n+1}, n\geq 1$, consider the following sequence of (pairs of) stopping times $(S_n^{k},T_n^{k})_{k\geq 1}$:
\begin{enumerate}
\item[] $S_n^1:=T_{G_n}, \;\;\mbox{the first hitting time of } G_n$, 
\item[] $T_n^1 := S_n^1+T_{E\setminus G}\circ \theta_{S_n^1}, \; \; \mbox{the first exit time from } G \mbox{ after } S_n^1$,
\item[] $S_n^k:=T_n^{k-1}+T_{G_n}\circ \theta_{T_n^{k-1}},\;\;\mbox{the first hitting time of } G_n \mbox{ afer } T_n^{k-1}$,
\item[] $T_n^k := S_n^k+T_{E\setminus G}\circ \theta_{S_n^k}  \; \; \mbox{the first exit time from } G \mbox{ after } S_n^k$.
\end{enumerate}
We define
\begin{align*}
I^n_G&:=(S_n^{1}, T_n^{1}) \cup (S_n^{2}, T_n^{2}) \cup \cdots \cup (S_n^{k}, T_n^{k})\cup \cdots\\
I_G&:=\mathop{\cup}\limits_n I^n_G.
\end{align*}
One can easily check that the definition of $I_G$ does not depend on the choice of the covering $G_n,n\geq 1$.
\end{defi}

\begin{rem} 
\begin{enumerate}[(i)]
    \item Note that $I^n_G\subset I^{n+1}_G, n\geq 1$ and that for each $\omega\in \Omega$, $I_G(\omega)$ is an open set in $[0,\infty)$ representing the union of all \textit{intervals of time on which the trajectory $X_{\cdot}(\omega)$ lies in $G$}. 
    \item Clearly, if $G=E$ then $I_G=(0,\infty)$.
    \item The reason of approximating $G$ from inside by $G_n,n\geq 1$ is the following: Suppose for simplicity that $G$ is a ball in $\mathbb{R}^d$ and $X$ is a $\mathbb{R}^d$-valued Brownian motion. Then $\partial G$ is made up of points which are regular for both $G$ and $\mathbb{R}^d \setminus G$, so that $T_{\mathbb{R}^d\setminus G} \circ \theta_{T_G}=0$ and the interval $(T_G,T_G+T_{\mathbb{R}^d\setminus G} \circ \theta_{T_G} )$ would be empty. However, this degeneracy can be avoided by counting the time spent in $G$ after the process (re)enters $G$, strictly, as when considering the above approximation of $G$ by $G_n,n\geq 1$. 
\end{enumerate}
\end{rem}

Further, by $G^r$ we be denote the set of all {\it regular} points of $G$, i.e. $x\in G^r$ if $\mathbb{P}^{x}(T_G=0)=1$.

The following definition settles the notion of diffusion which we shall use for the rest of the paper.
\begin{defi} \label{defi:diffusion}
Let $G\subset E$ be open. 
\begin{enumerate}[(i)]
\item For $x\in E$, we say that $X$ is a diffusion in $G$ under $\mathbb{P}^{x}$ if
\begin{equation} \label{diffusion}
\mathbb{P}^{x}(\{\omega : \emptyset \neq I_G(\omega)\cap [0, \zeta(\omega)) \ni t \longmapsto X_t(\omega) \mbox{ is continuous}\})=1.
\end{equation}
\item If (\ref{diffusion}) holds for all $x \in E$ ($\mathcal{U}$-a.e., $m$-a.e, resp. $m$-q.e.) then we say that $X$ is a diffusion in $G$ ($\mathcal{U}$-a.e., $m$-a.e, resp. $m$-q.e.).
\item When $G=E$ then in (i) and (ii) we simply say \say{diffusion} instead of \say{diffusion in $E$}.
\end{enumerate}
\end{defi}

If $G\subset E$ is (finely) open, let $X^G:=((X^G_t)_{t\geq 0},\mathbb{P}^x,x\in G)$ denote the right process on $G$ obtained by {\it killing} $X$ upon leaving the set $G$, given by
$$
X^G_t:=
\begin{cases}
X_t, &\mbox{if } t<T_{E\setminus G}\wedge \zeta\\
\Delta, & \mbox{otherwise}
\end{cases}.
$$
We remark that if we set 
$$G':=G^r\setminus (E\setminus G)^r,$$
then $G'$ is the largest finely open set such that $G\subset G'$ densely (w.r.t. the fine topology), in particular $G\subset G'\subset G^r$ and $T_{E\setminus G}=T_{E\setminus G'}$, hence $X^{G'}$ can be regarded as the {\it completion} (or {\it saturation}, in potential theoretic terms) of $X^G$ in $E$ by fine density.

We denote the resolvent of the killed process $X^G$ on $G$ by $\mathcal{U}^G:=(U^G_\alpha)_{\alpha>0}$, hence for all $f\in b\mathcal{B}$
\begin{equation} \label{eq:resolventG}
    U^G_\alpha f(x)=\mathbb{E}^x\left\{\int_0^{T_{E\setminus G}} f(X_t)\; dt \right\}, \quad x\in G.
\end{equation}

\begin{rem}
The right-hand side of \eqref{eq:resolventG} makes sense for all $x\in E$, hence we extend $U^G_\alpha$ to $E$ accordingly. In fact, $U^G_\alpha f$ extends from $G$ to $G'$ by fine continuity, whilst on $E\setminus G'$ it vanishes. Moreover, the following relation holds for all $f\in b\mathcal{B}$ and $\alpha>0$:
    \begin{equation} \label{eq:dynkin}
        U^G_\alpha f=U_\alpha f - B_{E\setminus G}^\alpha U_\alpha f \quad \mbox{ on }E.
    \end{equation}
\end{rem}

The following result is a key tool that one allows to reduce the study the diffusion property of $X$ in $G$, to the study of the diffusion property for the killed process $X^G$.
\begin{prop} \label{prop 2.7}
Let $G\subset E$ be open.
The following assertions are equivalent:
\begin{enumerate}
\item[(i)] The process $X$ is a diffusion in $G$ ($m$-q.e. on $E$ w.r.t. $\mathcal{U}$).
\item[(ii)] The killed process $X^{G}$ is a diffusion ($m$-q.e. on $G$ w.r.t. $\mathcal{U}^G$).
\item[(iii)] The killed process $X^{G'}$ is a diffusion ($m$-q.e. on $G'$ w.r.t. $\mathcal{U}^{G'}$).
\end{enumerate} 

\medskip
\noindent{Proof in \Cref{pf:1}.}
\end{prop}

Let us conclude this paragraph with the following useful result. It follows by combining \Cref{prop 2.7} with \Cref{prop 2.4}, (iii), so we skip its formal proof.
\begin{coro} \label{coro:m-ae}
Let $G\subset E$ be open.
The following assertions are equivalent:
\begin{enumerate}
\item[(i)] The process $X$ is a diffusion in $G$ $m$-q.e. w.r.t. $\mathcal{U}$.
\item[(ii)] The killed process $X^G$ is a diffusion (on $G$) $m$-a.e.
\end{enumerate}
\end{coro}

\begin{rem}
Let us emphasize that the proof of \Cref{coro:m-ae} required more effort in comparison to that of \Cref{prop2.6} mainly because in contrast to \Cref{prop 2.4}, (i), the function 
\begin{equation*}
    E\ni x\mapsto \mathbb{P}^{x}(\{\omega : \emptyset \neq I_G(\omega)\cap [0, \zeta(\omega)) \ni t \longmapsto X_t(\omega) \mbox{ is continuous}\})\in [0,1]
\end{equation*} 
is not necessarily excessive.
\end{rem}

%%%%%%%%%%%%%%% end paragraph
\paragraph{Resolvents,  generators, and martingale problems.}
In this paragraph we discuss several useful connections between resolvents of Markov kernels, generators, and corresponding martingale problems. 
%For clarity, we split this paragraph in two subparagraphs, depending on which framework we choose to work, a Dynkin type one or an $L^p$-one. 

Let us first give the following general definition which is going to be used throughout the entire paper:

\begin{defi} \label{defi:martingalepb}
Let ${\sf L}$ be a linear operator acting on a class $\mathcal{D}_0$ of real-valued $\mathcal{B}(E)$-measurable test functions such that for each $f\in\mathcal{D}_0$
\begin{equation*}
E \ni x\mapsto {\sf L} f(x)\in \mathbb{R} \mbox{ is } \mathcal{B}(E)\mbox{-measurable}.
\end{equation*}
We say that a Markov process $\left(\Omega, \mathcal{F}_t, X_t, \mathbb{P}^{x},x\in E\right)$ with lifetime $\zeta$ on $E$ solves the martingale problem associated ($m$-a.e., if $m$ is a given $\sigma$-finite measure on $E$) to $({\sf L},\mathcal{D}_0)$ if for each $f\in \mathcal{D}_0$
\begin{equation} \label{2.4}
f(X_{t\wedge \zeta})-\int_0^{t\wedge \zeta}{\sf L}f(X_r) \;dr, \quad t\geq 0
\end{equation}
is a $\mathcal{F}_t$-martingale w.r.t. $\mathbb{P}^{x}$, ($m$-a.e.) $x\in E$.

\end{defi}

Now, let $(U_\alpha)_{\alpha>0}$ be a resolvent of Markov kernels on $E$. 
Let $\alpha_0 >0$ such that the kernels $U_{\alpha}, \alpha>\alpha_0$ can be extended to bounded linear operators on $L^p(m)$ for some $1\leq p<\infty$; in particular, it is necessary that $m(A)=0$ implies $U_\alpha(1_A)=0$ $m$-a.e., $\alpha\geq 0$.
In this situation, we say that the resolvent of kernels $\mathcal{U}_{\alpha_0}$ can be extended to a resolvent on $L^p(m)$.

If there exists $\alpha_0 >0$ such that $\mathcal{U}_{\alpha_0}$ can be extended to a resolvent on $L^{p}(E, m)$ 
for some $p\geq 1$, we denote by $({\sf L}_p^m, D({\sf L}_p^m))$ the corresponding {\it generator} on $L^{p}(E,m)$ given by
\begin{align} \label{2.2}
& D({\sf L}_p^m):=\{U_\alpha f : f\in L^{p}(E,m)\}\\
& {\sf L}_p^m U_\alpha f:=\alpha U_\alpha f-f \mbox{ for al } f\in L^{p}(E,m), \quad \alpha > \alpha_0; \nonumber
\end{align}
recall that by the resolvent equation, the above definition does not depend on $\alpha$.

Concerning the existence of a measure $m$ such that $\mathcal{U}$ can be extended to a resolvent on $L^p(m)$, we can always rely on {\it potential measures} (i.e. measures of the type $\mu\circ U_\alpha$), employing the following known result (see e.g. \cite{RoTr07} and \cite{BeCiRo19}).

\begin{prop} \label{prop 2.10}
For any $\sigma$-finite measure $\mu$ on $E$ and $\alpha_0 >0$, $\mathcal{U}_{\alpha_0}$ extends to a strongly continuous resolvent on $L^p(E, \mu \circ U_{\alpha_0})$ for each $1\leq p < \infty$; in fact, if $(P_t)_{t\geq 0}$ denotes the corresponding semigroup regarded on $L^1(E, \mu \circ U_{\alpha_0})$, then $\|e^{-\alpha_0 t}P_t\|_{L^1}\leq 1, t\geq 0$. 
Moreover, if $\tau$ is a topology on $E$ which generates $\mathcal{B}$ and $\lim\limits_{\alpha\to\infty}\alpha U_\alpha f=f$ point-wise on $E$ for all $f\in bC(E)$, then choosing $(x_n)_{n\geq 1}$ to be a dense subset in $E$ and setting $\mu:=\mathop{\sum}\limits_{1\leq n<\infty} \frac{1}{2^n}\delta_{x_n}$, we additionally have that $\mu\circ U_{\alpha_0}$ has full topological support.
\end{prop}

Describing the domain of the generator on $L^p (\mu\circ U_\alpha)$ in the sense of \eqref{2.2} is not always an easy task.
A situation when we can easily get such information is when one starts from the martingale problem:
\begin{prop} \label{prop 2.11}
Suppose that $X$ is a Markov process with resolvent $\mathcal{U}$ such that $\mathcal{U}_{\alpha_0}$ can be extended to a strongly continuous resolvent of bounded operators on $L^{ p}(E, m)$, with generator $({\sf L}_p^{m}, D({\sf L}_p^{m}))$ given by \eqref{2.2}. 
If $X$ solves the martingale problem associated $m$-a.e. to $({\sf L}_0,\mathcal{D}_0)$ in the sense of \Cref{defi:martingalepb} and  $\mathcal{D}_0 \cup {\sf L}_0(\mathcal{D}_0)\subset L^{ p}(E, m)$, then
\begin{align*}
&\mathcal{D}_0 \subset D({\sf L}_p^{m})\\
&{\sf L}_p^{m} f= {\sf L}_0f \quad \mbox{ for all } f\in \mathcal{D}_0,
\end{align*}
i.e. $({\sf L}_p^{m}, D({\sf L}_p^{m}))$ is a (closed) extension of $({\sf L}_0,\mathcal{D}_0)$ on $L^{ p}(E, m)$.

\medskip
\noindent{Proof in \Cref{pf:1}.}
\end{prop}

\subparagraph{A change-of-measure lemma.}
For technical reasons regarding the proof of \Cref{coro:nonqr} below, it will be useful to be able to replace $m$ with some equivalent finite measure, without loosing information about the domain $D({\sf L}_p^m)$ of the generator ${\sf L}_p^m$ on $L^p (m)$.  
Fortunately, this is always possible due to the following simple yet general result, which will be employed several times later on.

\begin{lem} \label{lem 2.11} Let $0<\rho\in L^1(m)\cap L^\infty(m)$, $\alpha >\alpha_0$, and consider the measure $m_\alpha^\rho:=(\rho\cdot m)\circ U_\alpha$.

If $\mathcal{U}_{\alpha_0}$ extends to a strongly continuous resolvent on $L^p(m)$ for some $1\leq p < \infty$, then $m_\alpha^\rho$ is equivalent to $m$ and

\begin{align*}
D({\sf L}_p^m) &\subset D({\sf L}_1^{m_{\alpha}^\rho}),\\
{\sf L}_1^{m_{\alpha}^\rho} f&={\sf L}_p^m f \quad \mbox{ for all } f\in D({\sf L}_p^m).
\end{align*}

\medskip
\noindent{Proof in \Cref{pf:1}.}
\end{lem}

%%%%%%%%%%%%%%%% end section

\section{The main results}\label{s:main}

Having in mind \Cref{prop2.1}, throughout this section we assume that $X=(\Omega, \mathcal{F}, \mathcal{F}_t,\\ X_t, \theta_t, \mathbb{P}^x)$ is a right Markov process on a Lusin topological space $(E,\tau)$, with lifetime $\zeta$, which is c\` adl\` ag with respect to $\tau$; in particular, $\tau$ is a natural topology. 
Further, let $m$ be a $\sigma$-finite measure on $E$ such that the resolvent $\mathcal{U}$ of $X$ is strongly continuous on $L^{p}(m)$ for some $1\leq p<\infty$.
Also, we keep all the notations introduced in Section 2.

Let us first introduce some notions which are slight modifications of the usual ones.
\begin{defi} \label{defi:nest}
\begin{enumerate}[(i)]
\item An increasing sequence $(F_n)_{n\geq 1}$ of closed (respectively open) subsets of $E$ is called an 
{\rm $m$-nest}  of closed (respectively open) sets if
$$
\mathbb{P}^{x}\{\mathop{\sup}\limits_{n}T_{E\setminus {F}_{n}}\geq \zeta\}=1 \;\;m\mbox{-a.e.}
$$
\item A function $u:E\rightarrow \mathbb{R}$ is called {\rm $m$-quasi-continuous} (on short, $m$-q.c.) if there exists an $m$-nest of closed (or open) sets $(F_n)_{n\geq 1}$ such that $u|_{\overline{F_n}}$ is continuous for each $n\geq 1$.
\end{enumerate}
\end{defi}

Note that if $(F_n)_{n\geq 1}$ is an $m$-nest of open sets, then $(\overline{F_n})_{n\geq 1}$ becomes a $m$-nest of closed sets.
Also, let us give here two general lemmas that are going to be employed in the proof of the main result, namely \Cref{thm 2.2}.

\begin{lem} \label{lem:nest}
Suppose that $(F_n)_{n\geq 1}$ is an $m$-nest of closed (or open) sets and $G$ is a $\mathcal{B}$-measurable subset in $E$. 
If $G_n:= F_n \cap G$, then
\begin{equation*}
    \mathbb{P}^x\left\{\sup\limits_n T_{E\setminus G_n} \geq  T_{E\setminus G}\wedge \zeta \right\}=1, \quad m\mbox{-a.e. } x\in E.
\end{equation*}

\medskip
\noindent{Proof in \Cref{pf:2}.}
\end{lem}

\begin{lem} \label{lem:nestqe}
Suppose that $(F_n)_{n\geq 1}$ is an $m$-nest of closed (or open) sets, and in either situation set 
\begin{equation}
    v:=\inf\limits_n B^1_{E\setminus \overline{F_n}} \quad \mbox{on }E.
\end{equation}
Then the set $[v>0]$ is $m$-innesential.

\medskip
\noindent{Proof in \Cref{pf:2}.}
\end{lem}

The following definition settles the condition which is at the core of our main result.

\begin{defi}[$\rm{Loc}_m(G)$] \label{defi:locG}
For an open set $G\subset E$ we say that condition $\rm{Loc}_m(G)$ holds if there exist a sequence of functions $(f_n)_n\subset bp\mathcal{B}$ which are $m$-q.c. with some common $m$-nest of open (respectively closed) sets $(F_n)_{n\geq 1}$, and $(\varphi_n)_n\subset C(\mathbb{R}, \mathbb{R}_+)$ with the following properties:
\begin{enumerate}[(i)]
\item $(f_n)_n$ separates the points of $\mathop{\bigcup}\limits_{n \geq 1} F_n\cap \overline{G}$ in the sense that for every $G \ni y\neq x\in \mathop{\bigcup}\limits_{n \geq 1} F_n\cap \overline{G}$ there exists $n$ such that $f_n(x)<f_n(y)$.
\item $0\leq \varphi_k(x) \mathop{\nearrow}\limits_{k}x1_{[x>0]}=:x^{+}$ for all $x\in \mathbb{R}$ and
\begin{equation}
\mathcal{C}:=\{\varphi_k(f_n - \varepsilon) : n,k\geq 1, \varepsilon \in \mathbb{R}_+\}\subset D({\sf L}_p^m).
\end{equation}
\item For all $u\in \mathcal{C}$ we have ${\sf L}_p^m u=0$ $m$-a.e. on $\mathop{\wideparen{[u=0]}}\limits^\circ \cap G \cap F_n$ (respectively on $[u=0]\cap G\cap F_n$) for all $n\geq 1$.
\end{enumerate} 
\end{defi}
 
\begin{rem} \label{rem 3.2}
\begin{enumerate}[(i)]
\item If the functions $f_n,n\geq 1$ from \Cref{defi:locG} are continuous, the $m$-nest can simply be taken $F_n:=E,n\geq 1$. Also, we emphasize that the elements of the $m$-nest are not required to be compact, because the compactness is actually related to the fact that the process is c\` adl\` ag (which we already assumed), and not with the continuity of the trajectories. 
\item As already observed right after \Cref{defi:locG}, an $m$-nest of open sets generates an $m$-nest of closed sets if we take the closure of its elements. 
So working directly with $m$-nests of closed sets sounds more convenient. However, condition (iii) from \Cref{defi:locG} is sensitive to the two cases, being more relaxed in the case of an $m$-nest of open sets; in fact, in the case of an $m$-nest of closed sets, condition (iii) could pose difficulties if the topological boundary of $[u=0]\cap G \cap F_n$ is not negligible with respect to $m$. However, in our opinion this possible inconvenience is mostly theoretical; in fact, in practice we can frequently choose the functions $f_n,n\geq 1$, to be continuous on $E$, as detailed in the following two points.
\item Suppose we are in the case $E=\mathbb{R}^{d}$ and $G$ is an open subset of $E$.
For a function $f\in C_c^{\infty}(\mathbb{R}^d)$ such that $f(0)>0$, and a sequence $(x_n)_{n\geq 1} \subset G$ which is dense in $G$, set
\begin{equation*}
f_{n,k}(x):=f(n(x-x_k)), \; \; n,k\geq 1, x\in \mathbb{R}^{d}.
\end{equation*}
It is easy to see that for each $k\geq 1$ there exists $1\leq n(k)<\infty$ such that $f_{n,k}\in C_c^\infty(G)$ for all $n\geq n(k)$, and if $(f_n)_{n\geq 1}$ is a renumbering of $(f_{n,k})_{k\geq 1}^{n\geq n(k)}$, then $(f_n)_{n\geq 1}$ satisfies $\rm{Loc}_m(G)$, (i), for any $m$-nest.
\item If $E$ is a separable Banach space and $M \subset E$ is the state space of the Markov process under consideration, we can mimic the construction from (i) as follows: Let $(l_n)_{n\geq 1}\subset E'$ be total, i.e., if $x\in E$, $l_n(x)=0$ for all $n\geq 1$ implies $x=0$, and $(x_k)_{k\geq 1}$ be a dense subset in $M$.
For each $n\geq 1$, let $f_n\in C_c^{\infty}(\mathbb{R}^n)$ such that $f_n(0)>0$, and set
\begin{equation*}
f_{n,k,N}(x):=f_n(N(l_1(x-x_k),\cdots,l_n(x-x_k))), \; \; n,k\geq 1, x\in M.
\end{equation*}
Renumbering $(f_{n,k,N})_{n,k,N\geq 1}$ as $(f_n)_{n\geq 1}$, we have that for every $x\neq y$ from $M$, there exists $n\geq 1$ such that $f_n(x)<f_n(y)$.
\end{enumerate}
\end{rem}

The central result of this paper is the following.

\begin{thm} \label{thm 2.2} % Theorem 2.6
Let $G$ be an open subset of $E$. 
The following assertions hold.
\begin{enumerate}[(i)]
\item If condition $\rm{Loc}_m(G)$ is satisfied then $X$ is a diffusion in $G$ $m$-q.e.
%\begin{equation*}
%\mathbb{P}^{x}(\{\omega : I_G(\omega)\cap [0, %\zeta(\omega)) \ni t \longmapsto X_t(\omega) %\mbox{ is continuous}\})=1 \;\; m\mbox{-q.e.} %\ (\mbox{in } x)
%\end{equation*}
\item Conversely, if $X$ killed upon leaving $G$ is a diffusion on $G$ $m$-a.e., then for all $m$-q.c. functions $u\in D({\sf L}_p^m)$
\begin{equation*}
{\sf L}_p^m u=0 \;\; m\mbox{-a.e. on }\mathop{\wideparen{[u=0]}}\limits^\circ\cap G.
\end{equation*}
\end{enumerate}

\medskip
\noindent{Proof in \Cref{pf:2}.}
\end{thm}

\begin{rem}

\begin{enumerate}[(i)] \label{rem:topology}
    \item Theorem \ref{thm 2.2}, (i) remains true even if we drop the assumption that the resolvent $\mathcal{U}$ extends on $L^{p}(E,m)$, and we replace the generator $(L,D(L))$ as follows: Assume that $\mathcal{U}$ respects the $m$-classes and let $V\subset L^{0}(E,m)$ s.t. for some $\alpha >0$ we have $U_\alpha |f|<\infty$ $m$-a.e. for all $f\in V$. Consider $(L, D_\alpha(L; V))$ given by $D_\alpha(L; V):=\{U_\alpha f \;: \; f\in V\}$, $LU_\alpha f= \alpha U_\alpha f-f$ for all $f\in V$.
    \item Theorem \ref{thm 2.2}, (i) remains true if we drop the assumption that the process $X$ is c\` adl\` ag w.r.t. a Lusin topology on $E$, and assume instead that the topology is Lusin merely relative to each element $F_n$ of the $m$-nest that is worked with. This remark is in particular useful in infinite dimensional Banach spaces, in the case when the process $X$ is c\` adl\` ag merely with respect to the weak topology which is not Lusin because it is not metrizable; nevertheless it becomes Lusin relatively to any ball in the space. If a norm-like function is $\alpha$-excessive for some $\alpha \geq 0$, then one can deduce that the balls of radius $n\geq 1$ form a nest, hence taking into account the above remark our result could still be applied to deduce the diffusion property in the weak topology.
\end{enumerate}

\end{rem}

\paragraph{The measure-free counterpart of \Cref{thm 2.2}.}
Before proceeding to the main result of this paragraph, for the sake of comparison let us recall that following well known result.
\begin{prop}[cf. \cite{BlGe68}, Proposition 9.10] \label{prop:Getoor}
Suppose that $E$ is a second countable locally compact space and $X$ is a standard process with Feller transition function, i.e. $P_t(C_0(E))\subset C_0(E)$ and $\lim\limits_{t\to 0}P_tf=f$ uniformly on $E$.
If for any compact $K\subset E$ and any open neighbourhood $D$ of $K$ it holds that
\begin{equation} \label{eq:locGetoor}
    \lim\limits_{t\searrow 0}\frac{P_t(E\setminus G,x)}{t}=0 \quad \mbox{ uniformly on } K,
\end{equation}
then $X$ is a diffusion.
\end{prop}

\begin{rem}
If $u\in pC_0(E)$ and if we set $G:=\mathop{\wideparen{[u=0]}}\limits^\circ$ then uniformly on each compact subset of $G$, in particular pointwise on $G$, we have
\begin{equation}
    \lim\limits_{t\searrow 0} \frac{P_tu(x)-u(x)}{t}= \lim\limits_{t\searrow 0} \frac{P_tu(x)}{t}\leq \frac{|u|_\infty P_t(E\setminus G,x)}{t} =0.
\end{equation}
Hence if $\mathcal{L}u(x):=\lim\limits _{t\searrow 0} \frac{P_tu(x)-u(x)}{t}, x\in E$ whenever the limit exists, then $\mathcal{L}u=0$ on $\mathop{\wideparen{[u=0]}}\limits^\circ$; this makes perfect match with the key property iii) of $\rm {Loc}_m(G)$ in \Cref{defi:locG}.
\end{rem}

Our next aim is to show that \Cref{thm 2.2} can be easily employed to extend \Cref{prop:Getoor} in much more general settings.
To this end, assume that the process $X$ is a c\` adl\` ag right process on a Lusin topological space $E$, with resolvent $\mathcal{U}$ and transition function $(P_t)_{t\geq 0}$.
However, this time we are not given a measure $m$ on $(E,\mathcal{B})$ for which we can directly employ \Cref{thm 2.2}.
Instead, for $\alpha > 0$ consider the generator $(L, D_\alpha(L))$ given by:
\begin{align}
D_\alpha(L):=\big\{ &u \in \mathcal{B} : \lim\limits_{t\searrow 0} \frac{P_t u- u}{t} \mbox{ exists } \mathcal{U}\mbox{-a.e. on } E \mbox{ and  there exists } V_u\in p\mathcal{B}(E) \mbox{ such that } \nonumber\\
&U_\alpha V_u <\infty \mbox{ and } \left| \frac{P_t u- u}{t} \right|\leq V_u \;\; \mathcal{U}\mbox{-a.e. on } E\big\},\label{eq:dgen}
\end{align}
whilst
\begin{fleqn}
\begin{equation}
Lu(x):=  \lim_{t\searrow 0} \frac{P_t u(x)- u(x)}{t}, \quad u\in D_\alpha(L), x\in E \;\; \mathcal{U}\mbox{-a.e.}\label{eq:gen}
\end{equation} 
\end{fleqn}

\begin{rem}
If in the definition of the operator $(L, D_\alpha (L))$ we assume that $V_u $ is bounded, then $(L, D_\alpha (L))$ is similar to (yet still weaker than) the generator introduced by E. B. Dynkin, \cite{Dy65}, page 54; cf. also \cite{Fi88}. 
\end{rem}

We introduce the following pointwise version of property {$\rm{Loc}_m(G)$}:

\begin{defi}[$\rm{Loc}(G)$] \label{defi:pLocG}
For an open set $G \subset E$ we say that condition $\rm{Loc}(G)$ holds if there exists $(f_n)_n\subset C_b^{+}(E)$,  $(\varphi_n)_n\subset C(\mathbb{R}, \mathbb{R}_+)$ and $\alpha_0 >0$ with the following properties:
\begin{enumerate}[(i)]
\item $(f_n)_n$ separates the points of $\overline{G}$ in the sense that for every $\overline{G}\ni x\neq y\in G$ there exists $n$ s.t. $f_n(x)<f_n(y)$.
\item $0\leq \varphi_k(x) \mathop{\nearrow}\limits_{k}x1_{[x>0]}=:x^{+}$ for all $x\in \mathbb{R}$ and 
$$ 
\mathcal{C}:=\{\varphi_k(f_n - \varepsilon) : n,k\geq 1, \varepsilon \in \mathbb{R}_+\}\subset D_{\alpha_0} (L)
$$

\item $Lu=0$ on $\mathop{\wideparen{[u=0]}}\limits^\circ\cap G$ for all $u\in \mathcal{C}$.
\end{enumerate} 
\end{defi}

The main result of this paragraph is the following.
\begin{coro}\label{coro 2.6}
$(i)$ If $G$ is an open subset of $E$ for which condition $\rm{Loc}(G)$ is satisfied,
then the Markov process $X$ is a diffusion in $G$.

$(ii)$ Conversely, if the killed process $X^G$ upon leaving $G$ is a diffusion on $G$, then $Lu=0$ on $\mathop{\wideparen{[u=0]}}\limits^\circ\cap G$ \;  $\mathcal{U}$-a.e. for every $u\in D_\alpha(L)\cap C_b(E)$, and everywhere if $Lu$ is in addition a finely (lower or upper) semi-continuous function.

\medskip
\noindent{Proof in \Cref{pf:2}.}
\end{coro}

\subparagraph{The non-homogeneous case.}
Let us show that \Cref{coro 2.6} can be easily extended to non-homogeneous transition functions and their associated Markov processes.
So, let us assume that $(P_{s,s+t})_{s,t\geq 0}$ is the transition function of a non-homogeneous c\`adl\`ag Markov process $(X_t)_{t\geq 0}$ on a Lusin space $E$.
Then
$$
Q_tf(s,x):=P_{s,s+t}f(s+t,\cdot)(x) \mbox{ for all } x\in E, s,t \geq 0
$$
is the transition function of the process $Z:=((X_{u_t},u_t))_{t\geq 0}$ on the product space $E \times [0, \infty)$, where $(u_t)_{t\geq 0}$ is the uniform motion to the right.
Suppose that $Z$ has a version which is a right process, denoted also by $Z$; for example, one can show that this is always true if $P_{s,t}$ is Feller, applying \Cref{prop2.1} to $(Q_t)_{t\geq 0}$.
Let $\overline{\mathcal{U}}:=(\overline{U}_\alpha)_{\alpha>0}$ denote the resolvent associated to $(Q_t)_{t\geq 0}$ on $E\times [0,\infty)$ and consider the parabolic generator $(\Lambda, D_\alpha(\Lambda))$ on $E$ associated to $(P_{s,s+t})_{s,t\geq 0}$, $\alpha>0$ and $s\geq 0$:

\begin{align*}
    &\begin{aligned}
        D_\alpha(\Lambda):=\Big\{F\in b\mathcal{B}(E\times [0,\infty)) \;:\; &\left|\frac{d F}{dt}(\cdot,\cdot)\right|_{\infty}<\infty, \; \lim\limits_{t\searrow 0}\; \frac{P_{s,s+t} F(s,\cdot)- F(s,\cdot)}{t} \mbox{ exists } \overline{\mathcal{U}}\mbox{-a.e,} \\
        & \mbox{ and there exists } V_F \mbox{ such that } \overline{U}_\alpha(V_F)<\infty \mbox{ such that }\\
        & \left| \frac{Q_t F- F}{t} \right|\leq V_F \quad \overline{\mathcal{U}}\mbox{-a.e on } E\times [0,\infty) \Big\}
    \end{aligned}\\
    &\Lambda F(x,t):=\left(\frac{d}{dt}+L_t\right)F(x,t)=\frac{d F(x,\cdot)}{dt}(t)+L_tF(\cdot,t)(x), \;\; F\in D_\alpha(\Lambda), (x,t)\in E\times[0,\infty)\\
    \intertext{where}
    &L_s F(\cdot,s)(x):=  \lim_{t\searrow 0} \frac{P_{s,t+s} F(\cdot,s)(x)- F(x,s)}{t}, \quad F\in D_\alpha(\Lambda), (x,t)\in E\times [0,\infty). 
\end{align*}

\begin{coro}\label{coro:nonh}
Assume that $\rm{Loc}(G)$ given in \Cref{defi:pLocG} holds for $G$ with $(L,D_\alpha(L))$ replaced by $(L_s,D_\alpha(\Lambda))$ for every $s\geq 0$.
Then the homogeneous process $Z$ is a diffusion on $\widetilde{G}:=G\times [0,\infty)$.

\medskip
\noindent{Proof in \Cref{pf:2}.}
\end{coro}

\paragraph{Domination hypothesis and diffusions in other natural topologies.}

Throughout this paragraph we consider a slightly different framework than in the beginning of this section, namely we still assume that $X=(\Omega, \mathcal{F}, \mathcal{F}_t, X_t, \theta_t, \mathbb{P}^x)$ is a right Markov process on a Lusin measurable space $(E,\mathcal{B})$, with lifetime $\zeta$, but we do not fix a certain topology $\tau$ on $E$ with respect to which $X$ is c\` adl\` ag. 
Instead, we would like to address the following question: under which conditions it is true that if $X$ is a diffusion in $G$ $m$-q.e. with respect to one natural topology, then it automatically enjoys the same property with respect to all natural topologies?
In other words, we are interested in understanding when the diffusion property is stable under changing the (natural) topology.

\begin{prop} \label{prop:Dnatural}
Suppose that condition $\mathbf{(D_{m})}$ from Appendix holds and let $\tau$ and $\tau'$ be two natural topologies.
Then there exists a common nest of compact sets $(F_n)_{n\geq 1}$ such that $\tau|_{F_n}=\tau'|_{F_n}$ for all $n\geq 1$.
In particular, if $G$ is open (with respect to $\tau$ or $\tau'$), then:
\begin{enumerate}[(i)]
    \item $X$ is a diffusion in $G$ ($m$-q.e.) with respect to $\tau$ if and only if it is a diffusion in $G$ ($m$-q.e.) with respect to $\tau'$.
    \item Condition ${\rm Loc_m(G)}$ holds with respect to $\tau$ if and only if it holds with respect to $\tau'$.
\end{enumerate}

\medskip
\noindent{Proof in \Cref{pf:2}.}
\end{prop}

As an immediate consequence of \Cref{prop:Dnatural} and \Cref{thm 2.2}, we conclude this paragraph as follows:
\begin{coro}\label{coro:D_m}
Assume that condition $\mathbf{(D_{m})}$ is satisfied and let $G\subset E$ be open with respect to some natural topology. 
If ${\rm Loc_m(G)}$ holds, then $X$ is a diffusion in $G$ $m$-q.e. with respect to all natural topologies. 
\end{coro}

\section{Applications}
\label{s:applications}
\subsection{Jump-Diffusions on domains in Hilbert spaces}
Let $(H,\langle \cdot, \cdot \rangle)$ be a separable real Hilbert space endowed with the norm topology, and let
\begin{enumerate}
\item[-] $\sigma:[0, \infty)\times H \rightarrow L_s(H)$ be measurable, where $L_s(H)$ denotes the space of bounded and symmetric linear operators on $H$,
\item[-] $(A(t))_{t\geq 0}$ be a family of densely defined linear operators on $H$ such that the domains $D(A^\ast (t))$ of the adjoint operators $A^\ast(t),t\geq 0$ have the property that there exists a countable subset $(e_n)_{n\geq 1} \subset \mathop{\cap}\limits_{t\geq 0}D(A^\ast(t))$ which is total in $H$, and $[0, \infty)\ni t \rightarrow A^\ast(t)x \in H$ is measurable for any $x\in \mathop{\cap}\limits_{t\geq 0}D(A^\ast(t))$,
\item[-] $b:[0, \infty)\times D(b) \rightarrow H$ be measurable, where $D(b)\in \mathcal{B}(H)$,
\item[-] $n(t,x; dy)$ be a Levy measure on $H$ for each $(t,x)\in [0, \infty)\times H$ such that 
$$
[0, \infty)\times H\ni (t,x) \rightarrow \int_B 1\wedge |y|^{2}\; n(t,x; dy)\in [0,\infty)
$$
is measurable for all $B\in \mathcal{B}(H)$.
\end{enumerate}

Further, we set
\begin{align*}
\mathcal{F}&\mathcal{C}_0^{\infty}(H)\\
&:=\begin{cases}
\left\{ f:H\rightarrow\mathbb{R}, f(\cdot)=\varphi(\langle \cdot, e_1 \rangle, \dots, \langle \cdot, e_n \rangle)  \; : \;  \varphi \in C_c^{\infty}(\mathbb{R}^{n}), n \geq 1\right\}, & \mbox{ if } dim(H)=\infty \\
C_c^{\infty}(\mathbb{R}^d), & \mbox{ if } H=\mathbb{R}^d.
\end{cases}    
\end{align*}
and consider the time-dependent integro-differential operator given by
\begin{align*}
{\sf L}_tf(x):=&\frac{1}{2} {\rm Tr} [\sigma(t,x)^{2}D^{2}f(x)]+ \langle x , A^\ast (t) Df(x)\rangle + \langle b(t,x) , Df(x) \rangle\\
&+ \int_H \big [f(x+y)-f(x)-\frac{\langle y, D f(x)\rangle}{1+|y|^{2}}\big ] \; n(t,x; dy)
\end{align*}
for all $t\geq 0, x\in D(b), f\in \mathcal{F}\mathcal{C}_0^{\infty}(H)$.

Having in mind situations where one can solve the martingale problem associated to $L_t$ only for some allowed starting points from $D(b)$, let $I\subset [0,\infty)$ be an interval and $M\subset D(b)$ be a $\mathcal{B}(H)$-measurable set.

On the class of test functions
$$
\mathcal{D}:= \left\{ f:I\times H \rightarrow \mathbb{R} : f(t,x)=\varphi(t, \langle x, e_1 \rangle, \dots, \langle x, e_n \rangle),\; \varphi \in C_c^{\infty}(I\times\mathbb{R}^{n}), n \geq 1\right\}
$$
let us consider the parabolic operator associated to $L_t$ given by
$$
{\sf L}f=\left(\dfrac{\partial}{\partial t} +{\sf L}_t\right)f \quad \mbox{ for all } f\in \mathcal{D}.
$$
Set:
\begin{align*}
  &F:=\{(t,x)\in [0, \infty) \times \mathbb{R}^{d}\; : \; n(t,x;\; \cdot)=0\} \text{ and }\\
    &G:=\mathop{{F}}\limits^\circ \cap \left(I\times M\right), \text{ i.e. the interior of } F\cap (I\times M) \text{ relative to the trace topology on } I\times M.
\end{align*} 

\begin{coro} \label{coro 5.1}
Let $X$ be any c\`adl\`ag right Markov process on $I\times M$ with resolvent $\mathcal{U}$, and $\overline{m}$ a $\sigma$-finite measure on $I\times M$ such that:
\begin{enumerate}[(i)]
    \item For some $\alpha_0$ the resolvent $\mathcal{U}_{\alpha_0}$ can be extended to a strongly continuous resolvent on $L^p(I\times M, \overline{m})$ for some $1\leq p < \infty$,
    \item $\mathcal{D} \cup {\sf L}(\mathcal{D})\subset L^p(I\times M, \overline{m})$,
    \item $X$ solves the martingale problem associated $\overline{m}$-a.e. to $({\sf L}, \mathcal{D})$ in the sense of \Cref{defi:martingalepb}.
\end{enumerate}
Then $X$ is a diffusion in $G$ $\overline{m}$-q.e.

\medskip
\noindent{Proof in \Cref{pf:3}.}
\end{coro}

\begin{coro}\label{coro:4.2} Let $X$ be a c\`adl\`ag right Markov process on $I\times M$ with resolvent $\overline{\mathcal{U}}$, which is a solution to the martingale problem associated to $({\sf L}, \mathcal{D})$.
Suppose that the coefficients of ${\sf L}$ are such that $\overline{\mathcal{U}}_{\alpha_0}(|{\sf L}f|)<\infty$ for all $f\in \mathcal{D}$; this is always fulfilled if the coefficients are uniformly bounded in $(t,x)$.
Then $X$ is diffusion in $G$.

\medskip
\noindent{Proof in \Cref{pf:3}.}
\end{coro}

\subsection{Measure-valued branching processes}
Let $E$ be a Lusin topological space and denote by $M(E)$ the space of all finite (positive) measures on $E$, endowed with the weak topology and corresponding Borel $\sigma$-algebra denoted by $\mathcal{M}(E)$.
Let $\varphi$ be a branching mechanism, i.e.
\begin{equation*}
\varphi(x,\lambda)=-b(x)\lambda-c(x)\lambda^2+\int_0^{\infty} (1-e^{-\lambda u}-\lambda u)n(x,du), \quad x\in E, \lambda \geq 0,
\end{equation*}
where
\begin{enumerate}
    \item[-] $0\leq c$ and $b$ are bounded $\mathcal{B}(E)$-measurable functions,
    \item[-] $n$ is a (positive) kernel on $E$ such that $x\mapsto \int_0^\infty u\vee u^2 \; n(x, du)$ is bounded.
\end{enumerate}
Further, let $\xi(t)_{t\geq 0}$ be a standard (hence c\` adl\` ag) Markov process on $E$ with transition function $(P_t)_{t\geq 0}$, and consider the nonlinear semigroup $(V_t)_{t\geq 0}$ on $pb\mathcal{B}(E)$ given by 
\begin{fleqn}
\begin{equation}
   V_tf(x):=v(t,x), \quad x\in E, t\in [0,\infty), f\in pb\mathcal{B}(E), 
\end{equation}
\end{fleqn}
 where $v$ is the unique solution to the non-linear evolution equation in mild form
\begin{equation}
    v(t,x)=P_tf(x)+\int_0^t P_{s}(\varphi(\cdot,v(t-s,\cdot)))(x) \; ds, \quad t\geq 0,x\in E.
\end{equation}
Informally, if $A$ is the generator of $(P_t)_{t\geq 0}$, then $v(t,x)$ solves
\begin{equation}
    \left\{
    \begin{aligned}
     &\frac{dv}{dt}(t,x)=Av(t,x)+\varphi(x,v(t,x)) \\ 
     &u(0,x)=f(x)
    \end{aligned}\\
    \right. 
\end{equation}
A key feature of the nonlinear operators $f\mapsto V_tf$ is that they are negative definite, so they induce a linear semigroup of sub-Markovian kernels $(Q_t)_{t\geq 0}$ on $M(E)$ uniquely characterized by
\begin{equation}
    Q_t(e_f)(\mu)=e_{V_tf}(\mu), \quad \mu\in M(E), t\geq 0, f\in pb\mathcal{B}(E), 
\end{equation}
where $e_f:M(E)\rightarrow [0,1]$ is given by
\begin{fleqn}
    \begin{equation*}
       e_f(\mu)=e^{-\int_E f(x)\; \mu(dx)}, \quad \mbox{ for each } f\in pb\mathcal{B}(E).
    \end{equation*}
\end{fleqn}

Let $X:=(X_t)_{t\geqslant 0}$ be a c\` adl\` ag right process on $M(E)$ with transition function $(Q_t)_{t\geq 0}$; such a process always exists, see e.g. \cite{Fi88}, \cite{Li11}, \cite{Be11} and the references therein.
Recall  that $X:=(X_t)_{t\geqslant 0}$ is called {\it superprocess} and  $\xi(t)_{t\geqslant 0}$ is  the {\it spatial motion} of $X$.

As in \cite{Fi88}, let us consider the following generator of $X$:
\begin{align*}
 &D(A)=\left\{U_\alpha(f) :  f\in b\mathcal{B}(E) \mbox{ is finely continuous }, \alpha>0\right\}\\
 &AU_\alpha f:=\alpha U_\alpha f-f, \quad \mbox{ for each }U_\alpha f\in D(A).
\end{align*}
Now we are in the position give a description of the generator of $X$.
Set:
\begin{align*}
& D_{00}(L):=\left\{M(E)\ni\mu \mapsto \psi(\mu(f_1),\cdots,\mu(f_n))\in \mathbb{R} \;:\; n\geq 1, (f_i)_{1\leq i\leq n}\subset D(A), \psi\in C_0^\infty(\mathbb{R}^n) \right\}\\
&LF(\mu)=\int_E c(x)F''(\mu; x) \mu(dx)+\int_E [AF'(\mu;\cdot)(x)-b(x)F'(\mu; x)] \mu(dx)\\
&\quad \quad \quad \quad  +\int_E\int_0^{\infty}[F(\mu+u\delta_x)-F(\mu)-uF'(\mu;x)]n(x, du) \mu(dx), \quad F\in D_{00}(L), \mu \in M(E),
\end{align*}
where
\begin{equation} \label{eq:diff}
    F'(\mu;x)=\lim\limits_{h\to 0}\dfrac{F(\mu+h\delta_x)-F(\mu)}{h}, \quad \mu\in M(E), x\in E.
\end{equation}
Then, the following result was obtained in \cite{Fi88}:

\begin{thm}[cf. Theorem 4.1 from \cite{Fi88}]
If $F\in D_0(L)$ then
\begin{equation}
    \lim\limits_{t\searrow 0} \dfrac{Q_tF(\mu)-F(\mu)}{t}=LF(\mu), \quad \mu\in M(E).
\end{equation}
In particular,
$F(X_t)-F(X_0)-\int_0^tLF(X_s) ds, t\geq 0$
is an a.s. locally bounded c\` adl\` ag martingale.
\end{thm}

In addition, the following estimate on $LF$ can be easily deduced:
\begin{lem}\label{lem:bbound}
If $F\in D_{00}(L)$ then
\begin{equation}\label{eq:bbound1}
    |LF(\mu)|\leq c \mu(E)=c \mu(1), \quad \mu \in M(E),
\end{equation}
where $c$ is some constant that does not depend on $\mu$.
In particular, there exists a generic constant $c>0$ such that
\begin{equation}\label{eq:bbound2}
|(Q_tF(\mu)-F(\mu))/t|\leq c \mu(1),  \quad \mu\in M(E),
\end{equation}
so that $(L,D_{00}(L))\subset (L,D_\alpha(L))$ for some big enough $\alpha>0$, where the former generator is the one given by \eqref{eq:dgen}-\eqref{eq:gen} with $V_u(\mu):=c\mu(1), \mu\in M(E)$.

\medskip
\noindent{Proof in \Cref{pf:3}.}
\end{lem}

\begin{coro} \label{cor5.5}
Let $N:=\{x\in E : n(x, \cdot)=0\}$ and $\mathcal{U}$ be the resolvent of the base process $\xi$. 
The following assertion are equivalent:
\begin{enumerate}[(i)]
    \item There exists an open set $G\subset M(E)$ such that $X$ is a diffusion in $G$ in the sense of \Cref{defi:diffusion}.
    \item $n(x,\cdot)\equiv 0$ for every $x\in E$.
    \item The branching process $X$ is a diffusion (on $M(E)$).
\end{enumerate}

\medskip
\noindent{Proof in \Cref{pf:3}.}
\end{coro}

According to \cite{BeVr21},  $X:=(X_t)_{t\geqslant 0}$    is called {\it pure branching superprocess} provided that it has no spatial motion, that is, each point of $E$ is a trap for  $\xi(t)_{t\geqslant 0}$ (i.e., $\mathbb{P}^x (\xi_t=x)=1$ for all t$\geqslant 0$ and $x\in E$).

As an immediate consequence of Corollary \ref{cor5.5} we get the following:
\begin{coro} Assume that $X:=(X_t)_{t\geqslant 0}$   is a  pure branching superprocess.
Then $X$ is a diffusion on $M(E)$ if and only if $n\equiv 0$.
\end{coro}

\paragraph{Branching processes with interactions.}
Now we consider a generalization of $(L, D_{00}(L))$ that allow the branching particles to interact, more precisely their base movement and branching mechanism are allowed to depend on the current state of the population. 
One general and direct way to do this is to simply let the coefficients of $(L, D_{00}(L))$ depend on $\mu\in M(E)$. 
Proving existence and (especially) uniqueness of such measure-valued processes with interactions is highly non-trivial. 
However, recall that our purpose here is another one: Assuming that such a process exists (be it non-unique), that it is a (right) Markov process and has c\`adl\`ag paths, can we decide whether it has continuous paths or not, e.g. merely by looking at its generator through the martingale problem? 
Therefore, in contrast with the beginning of this section, in this paragraph we start with the generator in an abstract form, while the well-posedeness of the martingale problem shall be an assumption. 

We consider that $M(E)$ is endowed with a Lusin topology $\tau$, e.g. the weak topology. 
Let $\widetilde{\varphi}$ be a branching mechanism {\it with interaction}, i.e.
\begin{equation*}
\widetilde{\varphi}(x,\mu,\lambda)=-\widetilde{b}(x,\mu)\lambda-\widetilde{c}(x,\mu)\lambda^2+\int_0^{\infty} (1-e^{-\lambda u}-\lambda u)\widetilde{n}(x,\mu, du), \quad x\in E, \mu\in M(E), \lambda \geq 0,
\end{equation*}
where
\begin{enumerate}
    \item[-] $0\leq \widetilde{c}$ and $\widetilde{b}$ are bounded $\mathcal{B}(E)$-measurable functions,
    \item[-] $\widetilde{n}$ is a (positive) kernel on $E\times M(E)$ such that $(x,\mu)\mapsto \int_0^\infty u\vee u^2 \; \widetilde{n}(x,\mu, du)$ is bounded.
\end{enumerate}
Further, let $\widetilde{A}(\mu)_{\mu\in M(E)}$ be a family of linear operators defined on a common class of test functions $\widetilde{D}\subset \{f:E\rightarrow \mathbb{R} : f \mbox{ is } \mathcal{B}(E)\mbox{-measurable and bounded}\}$,
\begin{equation*}
    \mathcal{D}\ni f \mapsto \widetilde{A}(\mu)f \in b\mathcal{B}(E).
\end{equation*}
Also, we assume that $\widetilde{D}$ separates the measures from $M(E)$.
\begin{align*}
& \widetilde{D}_{00}(\widetilde{L}):=\left\{M(E)\ni\mu \mapsto \psi(\mu(f_1),\cdots,\mu(f_n)\in \mathbb{R} \;:\; n\geq 1, (f_i)_{1\leq i\leq n}\subset \widetilde{D}, \psi\in C_0^\infty(\mathbb{R}^n) \right\}\\
&\widetilde{L}F(\mu)=\int_E c(x,\mu)F''(\mu; x) \mu(dx)+\int_E [\widetilde{A}(\mu)F'(\mu;\cdot)(x)-b(x,\mu)F'(\mu; x)] \mu(dx)\\
&\quad \quad \quad \quad  +\int_E\int_0^{\infty}[F(\mu+u\delta_x)-F(\mu)-uF'(\mu;x)]n(x,\mu, du) \mu(dx)
\end{align*}
for $F\in \widetilde{D}_{00}(\widetilde{L}), \mu \in M(E)$.

\begin{coro}\label{coro:ibranch}
Assume that there exists a right Markov process $X:=(X_t)_{t\geq 0}$ on $M(E)$ with a.s. c\`adl\`ag paths with respect to $\tau$, such that
\begin{enumerate}[(i)]
    \item $X$ solves the martingale problem associated to $(\widetilde{L}, \widetilde{D}_{00}(\widetilde{L}))$ in the sense of \Cref{defi:martingalepb}.
    \item There exists $\alpha_0>0$ such that 
    \begin{equation*}
        \mathbb{E}^\mu\{X_t(1)\}\leq e^{\alpha_0t}\mu(1), \quad \mu\in M(E), t\geq 0.
    \end{equation*}
\end{enumerate}
Then, if $G$ denotes the interior (w.r.t. $\tau$) of the set $\{\mu \in M(E) : \int_E n(\cdot, \mu, \cdot) =0\}$, we have that $X$ is a diffusion in $G$.

\medskip
\noindent{Proof in \Cref{pf:3}.}
\end{coro}

%%%%%%%%%%%%%%%%%%%%%%%%%%%%

\subsection{Diffusions associated with generalized Dirichlet forms}
Following \cite{St99b}, let $(E,\mathcal{B}, m)$ be as in Section 2, and consider $(\mathcal{A}, \mathcal{V})$ a real valued coercive closed form on $\mathcal{H}:=L^{2}(E,m)$, i.e. $\mathcal{V})$ is dense in $\mathcal{H}$ and $\mathcal{A}:\mathcal{V}\times \mathcal{V}\rightarrow \mathbb{R}$ is a positive definite bilinear form s.t.
\begin{enumerate}[-]
\item $\mathcal{V}$ is a Hilbert space with inner product $\tilde{\mathcal{A}}_1(u, v):=\dfrac{1}{2}(\mathcal{A}(u,v)+\mathcal{A}(v,u))+\langle u, v\rangle_{\mathcal{H}}$, $u,v\in \mathcal{V}$.
\item the {\it weak sector condition} $|\mathcal{A}(u, v)|\leq K \tilde{\mathcal{A}}_1(u, u)^{\frac{1}{2}}\tilde{\mathcal{A}}_1(v, v)^{\frac{1}{2}}$ holds for some constant $K>0$ and all $u,v \in \mathcal{V}$.
\end{enumerate} 
In particular, $\mathcal{A}$ corresponds uniquely to a strongly continuous resolvent $\mathcal{U}:=(U_\alpha)_{\alpha>0}$ of contractions on $\mathcal{H}$, with generator $(L,D(L))$ given by
\begin{align*}
& D(L):=\{U_\alpha f : f\in \mathcal{H}\}\\
& L(U_\alpha f):=\alpha U_\alpha f-f \mbox{ for al } f\in \mathcal{H}.
\end{align*}
We are interested in non-sectorial perturbations of $L$: let $(\Lambda, D(\Lambda, \mathcal{H}))$ be a linear operator on $\mathcal{H}$ s.t.:
\begin{enumerate}[-]
\item $(\Lambda, D(\Lambda, \mathcal{H}))$ is the generator of a strongly continuous semigroup of contractions on $\mathcal{H}$, which can also be restricted to a strongly continuous semigroup on $\mathcal{V}$, with generator $(\Lambda, D(\Lambda, \mathcal{V}))$.
\end{enumerate}
Throughout we consider the embedding $\mathcal{V} \hookrightarrow \mathcal{H}\equiv \mathcal{H}'\hookrightarrow \mathcal{V}'$.
Then by \cite{St99b}, Lemma 2.3, $\Lambda: D(\Lambda, \mathcal{H}) \cap \mathcal{V} \rightarrow \mathcal{V}'$ is closable on $\mathcal{V}$, and the closure is denoted by $(\Lambda, \mathcal{F})$.
In particular, $\mathcal{F}$ is a Hilbert space w.r.t. the graph norm $|\cdot|_{{\mathcal{F}}}^{2}:=|\cdot|_{\mathcal{V}} + |{\Lambda} (\cdot)|_{\mathcal{V}'}$.
Also, the semigroup generated by the adjoint $(\hat{\Lambda},  D(\hat{\Lambda}, \mathcal{H}))$ of $(\Lambda, D(\Lambda, \mathcal{H}))$ on $\mathcal{H}$ can be extended to a strongly continuous semigroup on $\mathcal{V}'$.
The corresponding generator on $\mathcal{V}'$ denoted by $(\hat{\Lambda}, D(\hat{\Lambda}, \mathcal{V}'))$ is the dual of
$(\Lambda, D(\Lambda, \mathcal{V}))$.
In particular, $\hat{\mathcal{F}}:=D(\hat{\Lambda}, \mathcal{V}')\cap \mathcal{V}$ is a Hilbert space w.r.t. the graph norm $|\cdot|_{\hat{\mathcal{F}}}^{2}:=|\cdot|_{\mathcal{V}} + |\hat{\Lambda}( \cdot)|_{\mathcal{V}'}$.
By \cite{St99b}, Lemma 2.7 we have
$$
\langle \Lambda u, v \rangle = \langle u, \hat{\Lambda}v \rangle \quad \mbox{ for all } u\in \mathcal{F}, v\in \hat{\mathcal{F}},
$$
so the following {\it bilinear form associated with} $(\mathcal{A}, \mathcal{V})$ and $(\Lambda, D(\Lambda, \mathcal{H}))$ is well defined:
\begin{equation*}
 \mathcal{E}(u,v):=
 \begin{cases}
     \mathcal{A}(u, v)-\langle \Lambda u, v \rangle \quad \mbox {if } u\in \mathcal{F}, v\in \mathcal{V} \\
      \mathcal{A}(u, v)-\langle \hat{\Lambda} v, u \rangle \quad \mbox {if } u\in \mathcal{V}, v\in \hat{\mathcal{F}}
  \end{cases},
\end{equation*}
and $\mathcal{E}_\alpha (u,v):=\mathcal{E} (u,v) + \alpha \langle u , v \rangle_\mathcal{H}$ for all $\alpha>0$.
By \cite{St99b}, Proposition 3.4 there exist two resolvent families of continuous linear bijections $W_\alpha : \mathcal{V}'\rightarrow \mathcal{F}$ and $\hat{W}_\alpha: \mathcal{V}'\rightarrow \hat{\mathcal{F}}$ s.t.
\begin{equation} \label{eq 3.1}
\mathcal{E}_\alpha(W_\alpha f, v)=\langle f, v\rangle=\mathcal{E}_\alpha(v, \hat{W}_\alpha f) \mbox{ for all } f\in \mathcal{V}', v\in \mathcal{V}.
\end{equation}
Moreover, by \cite{St99b}, Proposition 3.6 the restrictions of $W_\alpha$ (resp. $\hat{W}_\alpha$) to $\mathcal{H}$, denoted by $G_\alpha$ (resp. $\hat{G}_\alpha$) are strongly continuous resolvents of contractions on $\mathcal{H}$, and $\hat{G}_\alpha$ is the adjoint of $G_\alpha$.
We denote by $(L^{\Lambda}, D(L^{\Lambda}))$ the infinitezimal generator associated to $(G_\alpha)_{\alpha >0}$ on $\mathcal{H}$.

\begin{lem} \label{lem 3.1}
It holds that $D(L_0)\cap D(\Lambda, \mathcal{H}) \subset D(L^{\Lambda})$ and 
$$
L^{\Lambda}u=L_0u+\Lambda u \mbox{ for all } u\in D(L_0)\cap D(\Lambda, \mathcal{H}).
$$

\medskip
\noindent{Proof in \Cref{pf:3}.}
\end{lem}

Further, it is assumed that $E$ is a Lusin topological space whose Borel $\sigma$-algebra is precisely $\mathcal{B}$.
Also $(\mathcal{E}, \mathcal{V})$ is supposed to be a {\it generalized Dirichlet form}, i.e. the resolvent $(G_\alpha)_{\alpha >0}$ is sub-Markovian, and also that it has associated a c\`adl\`ag right Markov process $X=(\Omega, \mathcal{F}, \mathcal{F}_t, X_t, \theta_t, \mathbb{P}^x)$ on $E$ with lifetime $\zeta$, in the sense that the resolvent of $X$ can be extended on $L^{2}(E, m)$ and it coincides with $(G_\alpha)_{\alpha >0}$ there.

Now, as a coroboration of Lemma \ref{lem 3.1} and Theorem \ref{thm 2.2}, we can conclude this subsection with the following:
\begin{coro} \label{coro:GDF}
Let $G\subset E$ be open and assume that $\rm Loc_m(G)$ (see \Cref{defi:locG}) holds with $D({\sf L}_p^m)$ from (ii) replaced by $D(L_0)\cap D(\Lambda, \mathcal{H})$, and ${\rm L}_p^m$ from (iii) replaced by $L_0+\Lambda$.
Then $X$ is a diffusion in $G$ $m$-q.e.
% \begin{equation*}
% \mathbb{P}^{x}(\{\omega : [0, \zeta(\omega)) \ni t\mapsto X_t(\omega) \mbox{ is continuous}\})=1 \;\; m\mbox{-q.e.}
% \end{equation*}
\end{coro}

\subsection{Diffusions associated with $L^p$-semigroups}

Recall that in the case when $(P_t)_{t\geq 0}$ is the transition function of a right process with state space  $E$
and   $m$ is a {\it subinvariant  measure} w.r.t.   $(P_t)_{t\geq 0}$, i.e., $\int_E P_t f d m \leq \int_E f d m$ for all $f\in p\cb$ and $t\geq 0$,
then   $(P_t)_{t\geq 0}$
induces such  a strongly continuous semigroup of sub-Markovian  contractions on  $L^p(E, m)$.
A converse of this statement was obtained in \cite{BeBoRo06}, 
in particular, it was given an answer to the following  question formulated by G. Mokobodzki
in 1991 and addressed in \cite{Dong98}: 
given a semi-Dirichlet form 
$(\mathcal{E}, D(\mathcal{E}))$    on $L^2(E,m)$, 
can  we find a Lusin topology on $E$ such that  $\mathcal{B}$ is its Borel  $\sigma$-algebra 
and $(\mathcal{E}, D(\mathcal{E}))$  is quasi-regular with respect to this topology, thus ensuring the association of a c\` adl\`ag right (in fact standard) process?

Now we readdress the problem raised by G. Mokobodzki by updating the required path regularity of the process as follows: Given a strongly continuous semigroup of sub-Markovian contractions $(P_t)_{t\geq 0}$ on $L^p(E,m)$, where $(E,\mathcal{B},m)$ is a $\sigma$-finite measure space, under which conditions is it possible to find a topology on $E$ and a right Markov process with continuous paths (i.e. a diffusion) which corresponds to $(P_t)_{t\geq 0}$?
It is important to mention here that we do not assume that the semigroup is quasi-regular in any sense.

To answer this question, let us fix $p \in [1, \infty)$ and $(P_t)_{t\geq 0}$ a strongly continuous semigroup of sub-Markovian  contractions on $L^p(E,m)$, 
where $(E,\cb)$ is a Lusin measurable space and $m$  is a $\sigma$-finite measure on $(E,\cb)$.
Let  $(V_\alpha )_{\alpha >0}$ be the sub-Markovian resolvent of contractions on $L^p(E,m)$ induced by $(P_t)_{t\geq 0}$ on $L^p(E, m)$, 
$ V_\alpha  =\int_0^\infty e^{-\alpha  t} P_t d t ,\,\,  \alpha >0$. 
An element $u\in L^p_+(E, m)$
is called $\beta$-{\it potential} provided that $\alpha
V_{\beta+\alpha} u\leq u$ for all $\alpha>0$.
Let  $\cp_\beta$ be the set of all $\beta$-potentials. 
It is known that  if $u,
u'\in \cp_\beta$, $u\leq u'$, then there exists $R_\beta (u-u')\in
\cp_\beta$, i.e. the {\it r\'eduite} of $u-u'$, defined by
$R_\beta (u-u')=\bigwedge \{ v\in \cp_\beta : $ $v\geq u-u'\}$; here
$\bigwedge$ denotes the infimum in $\cp_\beta$. 
An element $u\in
\cp_\beta$ is called {\it regular} if for every sequence
$(u_n)_n\subset \cp_\beta$ with $u_n\nearrow u$ we have
$R_\beta(u-u_n)\searrow 0$. 
Recall that by \cite{BeBoRo06}, Lemma 3.1, that $u\in \cp_\beta$ is regular if and only if $R_\beta(u-nV_nu)\mathop{\longrightarrow}\limits_n 0$; also, if $\cv_\beta=(V_{\beta+\alpha})_{\alpha >0}$
is the resolvent of a right process and 
$u$ is a $\cv_\beta$-excessive function,   $u < \infty$, and $u\in  L^p(E,m)$, 
then $u$ is regular if and only if there exists
a continuous additive functional whose potential equals $u$ $m$-a.e.

Consider the following  $L^p$-version of condition ${\bf (D_m)}$ from Appendix, for  $\beta>0$:

\medskip
\noindent
${\bf (D_m^p)}$ There exists $f_o\in L^p(E, m)$ strictly positive such that every $\beta$-potential dominated by $V_\beta f_o$ is regular.

\begin{coro} \label{coro:nonqr}
Let $p \in [1, \infty)$ and $(P_t)_{t\geq 0}$ be a  strongly continuous semigroup of sub-Markovian bounded operators on $L^p(E,m)$, 
where $(E,\cb)$  is a Lusin measurable space and $m$  is a $\sigma$-finite measure on $(E,\cb)$.
Let $({\sf L}_p^m, D({\sf L}_p^m))$ be the $L^p$-generator of 
$(P_t)_{t\geq 0}$.
Suppose that condition ${\bf (D_m^p)}$ as well as the following version of condition ${\rm Loc_m(G)},G=E$ are satisfied: 

\medskip
\noindent
$\rm\mathbf{\widetilde{Loc}_m}.$ There exist a sequence of  bounded $\mathcal{B}$-measurable functions $(g_n)_{n\geq 1}$ from $L^p(E, m)$ and $(\varphi_n)_{n\geq 1} \subset C(\mathbb{R}, \mathbb{R}_+)$ such that: 
\begin{enumerate}[(i)]
\item $(g_n)_{n\geq 1}$ separates the set of all positive finite measures on $E$. %, in the sense that if $\zeta \neq \eta$ are two such measures then there exists $n\geq 1$ with $\zeta(g_n)<\eta(g_n)$.  
\item $0\leq \varphi_k(x) \mathop{\nearrow}\limits_{k}x1_{[x>0]}=:x^{+}$ for all $x\in \mathbb{R}$ and
if for some $\beta >0$ we define $f_n:= V_\beta g_n, n\geq1$, then $0\leq f_n\in L^\infty(m)$ and
\begin{equation} \label{5.1}
\mathcal{C}:=\{\varphi_k(\pm f_n+\varepsilon) : n,k\geq 1,\varepsilon \in \mathbb{R}\}\subset D({\sf L}_p^m)\cup V_\alpha(L^\infty),
\end{equation}
for some (hence all) $\alpha>0$, where
$(V_\alpha)_\alpha$ is the resolvent of $(P_t)_{t\geq 0}$, which note that on $L^\infty$ is well defined for all $\alpha>0$.
\item $Lu=0$ $m$-a.e. on $[u=0]$ for all $u\in \mathcal{C}$.
\end{enumerate}

Then there exist a Lusin topological space $E_o$ with $E \subset  E_o$, $E \in \cb_o$
(the $\sigma$-algebra of all Borel subsets of $E_o$), 
$\cb = \cb_o |_E$, and a (Borel)  right process $X$ with state space $E_o$, 
which is a diffusion $m_o$-q.e.,
such that the  transition function of $X$ regarded as a family of operators on $L^p(E_o, m_o)$,  
coincides with $(P_t)_{t\geq 0}$,  where $m_o$ is the measure on 
$(E_o, \cb_o)$ extending $m$ with zero on $E_o\setminus E$.

\medskip
\noindent{Proof in \Cref{pf:3}.}
\end{coro}

\begin{rem}
Note that condition \eqref{5.1} is slightly different than condition ii) from \Cref{defi:locG}, in the sense that $\mathcal{C}$ is richer in \eqref{5.1}, in particular it could also contain elements that are lower bounded by a strictly positive constant, thus not belonging to $L^p(m)$ if $m(E)=\infty$. 
This is why we are lead to relax the inclusion $\mathcal{C}\subset D({\sf L}_p^m)$ to $\mathcal{C}\subset D({\sf L}_p^m)\cup V_\alpha(L^\infty)$.
\end{rem}

\section{Proofs}
\label{s:proofs}

In this section we collect the proofs of the results presented above, grouping them according to the corresponding sections.

\subsection{Proofs for \Cref{s:preliminaries}}
\label{pf:1}

\begin{proof}[\textit{\textbf{Proof of \Cref{prop2.1}}}]
By dominated convergence and the right-continuity of $X$ it follows that $\lim\limits_{\alpha \to \infty}\alpha {U}_\alpha f=f$ point-wise on $E$ for all $f \in \mathcal{C}$.
Hence we can apply \cite[Proposition 2.1]{BeRo11a} and \Cref{thm 4.15} from Appendix to construct a right Markov process $X'=(\Omega', \mathcal{F}', \mathcal{F}'_t, X'_t, \theta'_t, \mathbb{P'}^x)$ on a larger Lusin topological space $(E', \tau')$ such that $E\subset E'$ is $\mathcal{B}(E')$-measurable and $\tau'|_E \subset \tau$, whose resolvent denoted by $\mathcal{U}'$ is an extension of $\mathcal{U}$ from $E$ to $E'$, that is $(U_\alpha'f)|_E=U_\alpha(f|_E)$ for all $\alpha>0$ and $f\in b\mathcal{B}(E')$.
Let $D$ be the countable set of dyadics in $[0,\infty)$, and denote by $D_E$ the space of the restrictions to $D$ of all c\`adl\`ag paths from $[0,\infty)$ to $E$ (resp. $E'$).
Also, consider the product ${E'}^D$ endowed with the canonical $\sigma$-algebra.

Now, by \cite{DeMe78}, Chapter IV, pages 91-92 we obtain that $D_E$ is a measurable subset of ${E'}^D$, so that we can regard the law $\mathbb{P}^x\circ (X^x)^{-1}$ on $D_E$ as a probability on ${E'}^D$ supported on $D_E$.
But now, because any probability on ${E'}^D$ is determined by its finite dimensional marginals, we deduce that as laws on ${E'}^D$,
\begin{equation*}
    \mathbb{P}^x\circ (X^x)^{-1}=\mathbb{P'}^x\circ (X')^{-1} \quad \mbox{ for all } x\in E.
\end{equation*}
Therefore, $\mathbb{P'}^x\circ (X')^{-1}$ is also supported on $D_E$ for all $x\in E$, so under $\mathbb{P'}^x, x\in E$, the paths of $X'$ are restrictions to $D$ of c\`adl\`ag paths in $E$ w.r.t $\tau$. 
Since $\tau'|_E\subset \tau$, we deduce that the entire paths of $X'$ lie in $E$ and are c\`adl\`ag with respect to $\tau$. 
This means that the right process $X'$ can be restricted to $E$ and has c\`adl\`ag paths there.
\end{proof}

\begin{proof}[\textit{\textbf{Proof of \Cref{prop 2.4}}}]

(i). The proof follows by similar arguments as in \cite{BeRo11a}, Proposition 5.1, (i); see also \cite{Do05}.

(iii). Under the hypothesis we have that the finely open set $[v>0]$ is $\mathcal{U}$-negligible, hence empty (cf. Remark \ref{rem 2.2}, (i)), where $v$ is the $\mathcal{U}$-excessive function from (i).

The second part of the statement follows by the first one since, under the assumption, $m([v>0])$ implies that $[v>0]$ is $\mathcal{U}$-negligible.

(iii). It follows similarly to the proof of (ii), but this time employing Remark \ref{rem 2.2}, (ii).
\end{proof}

\begin{proof}[\textit{\textbf{Proof of \Cref{prop 2.7}}}]
Let us prove first the equivalence of the \say{pointwise} versions of the assertions.
Let $G_n,n\geq 1$ be a sequence as in \Cref{defi:IG}.

(i) $\Rightarrow$ (ii). If $x\in G$ then one can easily verify that $\mathbb{P}^x$-a.s. we have $\inf\limits_n S_n^1=0$ and $T_n^1=T_{E\setminus G}, n\geq 1$, hence $(0, \zeta\wedge T_{E\setminus G})\subset I_G\cap[0,\zeta)$.  

(ii) $\Rightarrow$ (i). Let $d$ be a metric on $E$ which generates the topology.
By the strong Markov property, we get for each $k,n\geq 1$
\begin{flalign*}
\quad 0&\leq  \mathbb{E}^{x}\left\{ \sum\limits_{S_n^k<s< T_n^k}d(X_s,X_{s-}) \right \} = \mathbb{E}^{x}\left\{ \sum\limits_{S_n^k<s< T_n^k}d(X_s,X_{s-});  S_n^k<\infty\right \}\\
&=\mathbb{E}^{x}\left\{\mathbb{E}^{X_{S_n^{k}}}\left\{\sum\limits_{0<s< T_{E\setminus G}} d(X_s,X_{s-})\right\} ; S_n^{k}<\infty\right\}\\
&= 0,
\end{flalign*}
where the last equality follows from assumption (ii) and the fact that due to the right continuity of $X$, $X_{S_n^{k}}\in \overline{G}_n \subset G$ on $S_n^k<\infty$ $\mathbb{P}^{x}$-a.s. for all $x\in E$.
In other words, we obtained for each $k\geq 1$
\begin{equation*}
\mathbb{P}^{x}\left\{\omega :  X_{t-}(\omega)\neq X_t(\omega) \mbox{ for some } t\in (S_n^k(\omega), T_n^k(\omega))\cap [0, \zeta(\omega)\neq \emptyset\right\}=0,
\end{equation*}
so assertion (i) follows by the fact that $I_G=\mathop{\bigcup}\limits_n\mathop{\bigcup}\limits_k \left(S_n^k(\omega), T_n^k(\omega)\right)$.

(ii) $\Rightarrow$ (iii). To this end, assume that (ii) holds and recall that the process $X^{G'}$ killed upon leaving $G':=E\setminus (E\setminus G)^r$, is a right Markov process on $G'$, so by Proposition \ref{prop 2.4}, (i), the function
$$
v_{G'}(x):= \mathbb{P}^{x} (\{\omega : [0, \zeta\wedge T_{E\setminus G}(\omega)) \ni t \longmapsto X_t(\omega) \mbox{ is not continuous} \} ), \;\; x\in G',
$$
is excessive w.r.t. $X^{G'}$, hence finely continuous on $G'$. 
Since $G$ is finely dense in $G'$ and because $v_{G'}=v_G=0$ on $G$, it follows that $v_{G'}=0$ on $G'$. 

Since $T_{E\setminus G}=T_{E\setminus G'}$ $\mathbb{P}^x$-a.s. for all $x\in G$, the implication (iii) $\Rightarrow$ (ii) is clear.

Let us now deal with the $m$-q.e. version of the equivalences:

(i) $\Rightarrow$ (ii). Suppose that $N\subset E$ is an $m$-inessential set (w.r.t. $\mathcal{U}$) such that  \eqref{diffusion} holds for all $x\in E\setminus N$. 
Then it follows immediately that $N\cap G$ is $m$-inessential w.r.t. $\mathcal{U}^G$ on $G$, hence this implication is proved.

(ii) $\Rightarrow$ (i). The converse of the above implication is not as trivial as the direct one because in general it is not true that an $m$-inessential set w.r.t. $\mathcal{U}^G$ is also an $m$-inessential set w.r.t. $\mathcal{U}$.  
Nevertheless, by \Cref{prop 2.4}, the function \begin{equation*}
    G\ni x\mapsto v_G(x):= \mathbb{P}^{x} (\{\omega : [0, \zeta_G(\omega)) \ni t \longmapsto X_t(\omega) \mbox{ is not continuous} \} )
\end{equation*}
is  $\mathcal{U}^G$-excessive.
Now, if (ii) holds then clearly $m([v_G<1])=0$.
Also, since $v_G$ is $\mathcal{U}^G$-excessive, the set $[v_G>0]$ is finely open w.r.t. $\mathcal{U}^G$. 
But a subset of $G$ which is finely open w.r.t. $\mathcal{U}^G$ is also finely open w.r.t. to $\mathcal{U}$, hence $[v_G>0]$ is finely open w.r.t. $\mathcal{U}$ and $m$-negligible. 
Consequently, by \cref{rem 2.2} we have that $[v_G>0]$ is $m$-polar, hence by the same \cref{rem 2.2} it is contained in an $m$-inessential set (w.r.t. $\mathcal{U}$). 

(ii) $\Leftrightarrow$ (iii). The implication (iii) $\Rightarrow$ (ii) follows immediately since $[V_G>0]\subset [v_{G'}>0]$. 
The converse follows by applying (i) $\Rightarrow$ (ii) with $E$ replaced by $G'$, noticing that in this case we have $I_G\cap [0, \zeta_{G})=(0,\zeta_{G'})$.
\end{proof}

\begin{proof}[\textit{\textbf{Proof of \Cref{prop 2.11}}}]
Since $\mathcal{U}_{\alpha_0}$ is strongly continuous on $L^p(m)$, so is $(P_t)_{t\geq 0}$, and the generator given by \eqref{2.2} can also be described by
\begin{align*}
& D({\sf L}_p^{m}):=\left\{f\in L^{p}(E,m) : \lim\limits_{t\to 0} \frac{P_tf-f}{t} \mbox{ exists in } L^p(E,m)\right\}\\
& {\sf L}_p^{m}(f):= \lim\limits_{t\to 0} \frac{P_tf-f}{t} \mbox{ for al } f\in D(L_p^{m}).
\end{align*}
So, if $f\in \mathcal{D}_0$, by taking expectations in \eqref{2.4} we have
\begin{equation*}
P_tf(x)-f(x)=\mathbb{E}^{x}\left\{\int_0^{t\wedge \zeta}{\sf L}_0 f(X_s) \;ds\right\} = \int_0^t P_s{\sf L}_0f(x)\; ds\quad m\mbox{-a.e.},
\end{equation*}
hence
$$
{\sf L}_p^mf=\lim\limits_{t\to 0} \frac{P_tf-f}{t}=\lim\limits_{t\to 0}\frac{1}{t}\int_0^t P_s{\sf L}_0f \; ds={\sf L}_0f \quad \mbox{ in } L^p(E, m).
$$
\end{proof}

\begin{proof}[\textit{\textbf{Proof of \Cref{lem 2.11}}}]
First of all, note that $m_{\alpha}^\rho << m$ and $L^p(m)\subset L^1(m_{\alpha}^\rho)$, since 
$$
\int_E |f| \; dm_{\alpha}^\rho = \int_E \rho U_{\alpha}|f| \; dm \leq |\rho|_{L^{p^\ast}(m)}|U_{\alpha}|f||_{L^p(m)}<\infty,
$$
where $\frac{1}{p}+\frac{1}{p\ast}=1$.

To show that $m<<m_{\alpha}^\rho$, if $m_{\alpha}^\rho(A)=0$ then $U_\alpha 1_A = 0 \;m$-a.e. for all $\alpha>0$.
Therefore, $\lim\limits_{\alpha\to \infty} \alpha U_\alpha 1_A =1_A$ in $L^p(m)$ and $m(A)=0$.

Now, if $f\in D({\sf L}_p^m)$, there exists $g\in L^p(m)$ s.t. $f=U_{\alpha}g$ in $L^p(m)$, hence by the above inclusion, the equality holds also in $L^1(m_{\alpha}^\rho)$.
The statement now follows from definition \eqref{2.2} and \Cref{prop 2.10}
\end{proof}

\subsection{Proofs for \Cref{s:main}}
\label{pf:2}

\begin{proof}[\textit{\textbf{Proof of \Cref{lem:nest}}}]
Clearly, $T_{E\setminus G_n} \leq T_{E\setminus G}$ hence $\sup\limits_n T_{E\setminus G_n} \leq T_{E\setminus G}$ a.s.
If $\sup\limits_n T_{E\setminus G_n}(\omega) < T_{E\setminus G}(\omega)$ it means that $\zeta(\omega)\leq\sup\limits_n T_{E\setminus F_n}(\omega)<T_{E\setminus G}(\omega)$, hence the result follows.
\end{proof}

\begin{proof}[\textit{\textbf{Proof of \Cref{lem:nestqe}}}]
Clearly, since $(F_n)_{n\geq 1}$ is an $m$-nest, we get that $m([v>0])=0$, so it remains to prove that $\mathbb{P}^x(T_{[v>0]}=\infty)=1$ for all $x\in [v=0]$, or equivalently, that $B^1_{[v>0]}=0$ on $[v=0]$.
But the last equality is clearly satisfied if $R^1_{[v>0]}=0$ on $[v=0]$, so let us prove that this latter property holds.
In fact, because $[v>0]=\mathop{\cup}\limits_k [v>1/k]$ hence $\lim\limits_kR^1_{[v>1/k]}=R^1_{[v>0]}$, it is sufficient to prove that $R^1_{[v>1/k]}=0$ on $[v=0]$ for all $k\geq 1$. 
To this end, note that $v_n:=B^1_{E\setminus \overline{F_n}}$ is $1$-excessive for every $n\geq 1$, hence
\begin{align*}
    R^1_{[v>1/k]}&=\inf\{w\;:\; w \mbox{ is } 1\mbox{-excessive and  }w\geq 1 \mbox{ on } [v>1/k]\}\\
    &\leq \inf\limits_n kv_n=0 \mbox{ on } [v=0],
\end{align*}
which completes the proof.
\end{proof}

\begin{proof}[\textit{\textbf{Proof of \Cref{thm 2.2}}}]
(i).  Recall that we set $\zeta_G:=\zeta\wedge T_{E\setminus G}$. 
By \Cref{coro:m-ae}, it is sufficient (also necessary) to prove that the killed process $X^G$ is a diffusion $m$-a.e., that is
\begin{equation*}
\mathbb{P}^{x}(\{\omega : [0, \zeta_G(\omega)) \ni t \longmapsto X_t(\omega) \mbox{ is continuous}\})=1 \;\; m\mbox{-a.e.} \mbox{ on } G
\end{equation*}
We show this in five steps.

\paragraph{Step I.} Let $(F_i)_{i\geq 1}$ be a common $m$-nest of open (respectively closed sets) for $(f_n)_{n\geq 1}$ and set $G_i:=F_i\cap G, i\geq 1$.
By \Cref{lem:nest} and the remark just above it, we have $m$-a.e. on $G$ that
\begin{align*}
&\mathbb{P}^{x}\left\{\omega :  X_{t-}(\omega)\neq X_t(\omega) \mbox{ for some } t\in (0, \zeta_G(\omega))\right\}\\
&=\mathbb{P}^{x}\left\{\mathop{\bigcup}\limits_{i\geq 1} \{\omega : X_{t-}(\omega)\neq X_t(\omega) \mbox{ for some } t\in (0, T_{E\setminus G_i}(\omega))\}\right\}\\
&\leq\mathbb{P}^{x} \left\{\mathop{\bigcup}\limits_{i\geq 1}\mathop{\bigcup}\limits_{n\geq 1}\mathop{\bigcup}\limits_{s\in \mathbb{Q}_+}\mathop{\bigcup}\limits_{\varepsilon\in \mathbb{Q}_+}\Omega_{i,n,s,\varepsilon}\right\},
\end{align*}
where
\begin{equation*}
\Omega_{i,n,s,\varepsilon}:=\{f_n(X_s)< \varepsilon,\;  f_n(X_{s+T_{[f_n\geq\varepsilon]}\circ \theta_s}) > \varepsilon,\; X_s\in G_i, \; T_{[f_n\geq\varepsilon]}\circ \theta_s < (T_{E\setminus G_i}\circ \theta_s\}.
\end{equation*}
Indeed, the equality is clear, and if $\overline{F_i}\ni X_{t-}(\omega)\neq X_{t}(\omega)\in G_i$ for some $i$, then by condition $\rm{Loc}_m(G)$, (i), 
there exists $n\geq 1$ and $\varepsilon \in \mathbb{Q}$ s.t. $f_n(X_{t-}(\omega))<\varepsilon < f_n(X_{t}(\omega))$.
Since $f_n$ is continuous on each $\overline{F_i}$, there exists $s\in \mathbb{Q}_+$ s.t. $f_n(X_{s}(\omega))<\varepsilon$ 
and $t$ is the first time when $X_\cdot(\omega)$ hits $[f_n\geq \varepsilon]$ (but also $[f_n> \varepsilon]$) after time 
$s$, i.e. $t=s+T_{[f_n\geq \varepsilon]}\circ \theta_s(\omega)$, which proves the assertion.

\paragraph{Step II.} By Step I, it is sufficient to prove that for each $n\geq 1$, $s>0$ and $\varepsilon\in \mathbb{R}_+$ we have $m$-a.e. on $G$
\begin{equation} \label{eq 2.1}
\mathbb{P}^{x}\left(\left\{f_n(X_s)< \varepsilon, \; f_n(X_{s+T_{[f_n\geq\varepsilon]}\circ \theta_s}) > \varepsilon,\; X_s\in G_i, \; T_{[f_n\geq\varepsilon]}\circ \theta_s < T_{E\setminus G_i}\circ \theta_s\right\}\right)=0.
\end{equation}
But by the strong Markov property
\begin{align*}
\mathbb{P}^{x}\Bigl(\Bigl\{f_n(X_s)< \varepsilon, \; &f_n\left(X_{s+T_{[f_n\geq\varepsilon]}\circ \theta_s}\right) > \varepsilon,\; X_s\in G_i, \; T_{[f_n\geq\varepsilon]}\circ \theta_s < T_{E\setminus G_i}\circ \theta_s \Bigr\}\Bigr)\\
&=\mathbb{E}^{x}\left\{1_{[f_n< \varepsilon]\cap G_{i}}(X_s) \; \mathbb{E}^{X_s}\left\{1_{[f_n > \varepsilon]}(X_{T_{[f_n\geq\varepsilon]}})\; ; \;T_{[f_n\geq\varepsilon]} < T_{E\setminus G_i}\right\}\right\},
\end{align*}
hence \eqref{eq 2.1} holds if $\mathbb{E}^{x}\left\{1_{[f_n > \varepsilon]}(X_{T_{[f_n\geq\varepsilon]}})\; ; \; T_{[f_n\geq\varepsilon]} < T_{E\setminus G_i}\right\}=0$ $m$-a.e. 
on $[f_n< \varepsilon] \cap G_i$, which is in turn true if 
\begin{equation} \label{eq 2.2}
B^{1}_{[f_n\geq\varepsilon]\cup (E\setminus G_i)}(f_n-\varepsilon)^{+}=0 \;m\mbox{-a.e. on } [f_n<\varepsilon]\cap G_i.
\end{equation}

\paragraph{Step III.} Since $(f_n-\varepsilon)^{+}=\mathop{\sup}\limits_k \varphi_k(f_n-\varepsilon)$ we get 
\begin{equation*}
B^{1}_{[f_n\geq\varepsilon]\cup (E\setminus G_i)}(f_n-\varepsilon)^{+}= \lim\limits_k B^{1}_{[f_n\geq\varepsilon]\cup (E\setminus G_i)}\varphi_k(f_n-\varepsilon) \;\; m\mbox{-a.e.},
\end{equation*}
hence relation \eqref{eq 2.2} is true if for each $k\geq 1$
\begin{equation} \label{eq 2.3}
B^{1}_{[f_n\geq\varepsilon]\cup (E\setminus G_i)}\varphi_k(f_n-\varepsilon)=0 \mbox{ on } [f_n<\varepsilon]\cap G_i \;\; m\mbox{-a.e.}
\end{equation}

\paragraph{Step IV.} Now we show that \eqref{eq 2.3} holds, which completes the proof.
Let $u:=\varphi_k(f_n-\varepsilon)\in \mathcal{C}$ and $f\in L^{p}(E,m)$ such that $u=U_1f$ $m$-a.e.
By condition $\rm{Loc}_m(G)$, (iii), we have 
\begin{equation*}
f=Lu-u=0\;\; m\mbox{-a.e. on } \mathop{\wideparen{[u=0]}}\limits^\circ\cap G_i \;(\mbox{respectively on } [u=0]\cap G_i).
\end{equation*}
Now, the trick is to choose $f^{\ast}:E\rightarrow \mathbb{R}$ a measurable version of $f$ such that 
\begin{equation*}
f^{\ast}=0 \; \mbox{ pointwise on } \; \mathop{\wideparen{[u=0]}}\limits^\circ\cap G_i \;(\mbox{respectively on } [u=0]\cap G_i).
\end{equation*}
Note that $U_1f^\ast$ is defined merely $m$-a.e. and $U_1f^{\ast}=u$ $m$-a.e., but in order to rigorously be able to replace $u$ with $U_1f^{\ast}$ in \eqref{eq 2.3}, as we plan to do in the sequel, we need the previous equality to hold $m$-q.e.
To show that this is indeed the case, let $v$ be the function defined in \Cref{lem:nestqe} and consider the sets $[U_1|f^{\ast}|=\infty]$ and $[v>0]$ which are both $m$-innesential. 
Then set
\begin{equation*}
E_0:=[U_1|f^{\ast}|<\infty]\cap[v=0], 
\end{equation*}
so that $E\setminus E_0$ is also $m$-inessential.
In particular, $E_0$ is finely open and
\begin{enumerate}
    \item[-] $\mathbb{P}^x\{T_{E\setminus E_0}=\infty\}=1 \mbox{ for all } x\in E_0$,
    \item[-] $\mathbb{P}^x\{\sup\limits_nT_{E_0\setminus F_n}\geq \zeta\}=1, \quad x\in E_0$,
    \item[-] $u|_{E_0} \mbox{ and } U_1f^\ast|_{E_0} \mbox{ are finely continuous on } E_0$,
\end{enumerate}
where the second property follows by \Cref{lem:nestqe}.
Therefore, by taking the restriction of $X$ and $\mathcal{U}$ to $E_0$ and since $m([|u|_{E_0}-(U_1f^\ast)|_{E_0}|>0])=0$, we can apply \Cref{rem 2.2}, (ii) on $E_0$ to deduce that that $u|_{E_0}=(U_1f^\ast)|_{E_0}$ $m$-q.e. on $E_0$, and finally, since $E\setminus E_0$ is $m$-innesential, that 
\begin{equation*}
u=U_1f^{\ast}\; m\mbox{-q.e. on E}
\end{equation*}
Let us generically set $A:=[f_n\geq\varepsilon]\cup (E\setminus G_i)$.
Using first Remark \ref{rem 2.2}, (iv) and then the strong Markov property, we get that for $m$-a.e. $x\in E_0\cap G$
\begin{fleqn}
\begin{align*}
B_{A}^{1}&u(x)=B_{A}^{1}(U_1f^{\ast})(x)=\mathbb{E}^{x}\{e^{-T_{A}}U_1f^{\ast}(X_{T_{A}})\}\\
&=\mathbb{E}^{x}\bigg\{e^{-T_{A}}\mathbb{E}^{X_{T_{A}}}\bigg\{\int_0^{\infty}e^{-t}f^{\ast}(X_t)\; dt\bigg\}\bigg\}=\mathbb{E}^{x}\bigg\{e^{-T_{A}}\int_0^{\infty}e^{-t}f^{\ast}(X_{t+T_{A}})\; dt\bigg\}\\
&=\mathbb{E}^{x}\bigg\{\int_0^{\infty}e^{-(t+T_{A})}f^{\ast}(X_{t+T_{A}})\; dt\bigg\}=\mathbb{E}^{x}\bigg\{\int_{T_{A}}^{\infty}e^{-t}f^{\ast}(X_{t})\; dt\bigg\}\\
&=U_1f^{\ast}(x)-\mathbb{E}^{x}\bigg\{\int_0^{T_{A}}e^{-t}f^{\ast}(X_{t})\; dt\bigg\}=0 \mbox{ on } [f_n<\varepsilon]\cap G_i,
\end{align*}
\end{fleqn}
because $f^{\ast}=0$ pointiwse on $\mathop{\wideparen{[u=0]}}\limits^\circ\cap G_i\supset [f_n<\varepsilon]\cap G_i$ (respectively on $[u=0]\cap G_i\supset [f_n<\varepsilon]\cap G_i$).
Note that all the above expressions involving $f^{\ast}$ make sense on $E_0$.

The second part of assertion (i) follows by Proposition \ref{prop 2.4}, (iii).

\vspace{0.2 cm}
(ii).
Assume now that the killed process $X^G$ is a diffusion on $G$ $m$-a.e. 
Let $u\in  D(L)$ be $m$-q.c. and $f:E\rightarrow \mathbb{R}$ $\mathcal{B}$-measurable s.t. $f\in L^{p}(E,m)$ and $u=U_1f$ $m$-a.e.
Clearly, 
\begin{equation*}
    Lu=0 \; m\mbox{-a.e. on } \mathop{\wideparen{[u=0]}}\limits^\circ\cap G:=D \mbox{ if and only if } f=0  \; m\mbox{-a.e. on } D.
\end{equation*}

Let $\mathcal{U}^D$ denote the resolvent of the killed process upon leaving $D$, given by \eqref{eq:resolventG}. 
Now, as in Step IV, if we set $ E_0:=[U_1|f|<\infty]\cap [v=0]$ where $v$ is given by \Cref{lem:nest} for some $m$-nest $(F_n)_{n\geq 1}$ attached to $u$, then $u=U_1f$ $m$-q.e. on $E$.
Moreover, one can easily observe that $(0, \zeta_G) \ni t\mapsto u(X_t) \in \mathbb{R}$ is continuous $\mathbb{P}^x$-a.s. $m$-a.e. $x\in G$, and consequently that $B_{E\setminus D}^{1}u=0$ $m$-a.e. on $G$.
Therefore,
$$
U_1^Df(x)=u(x)-B_{E\setminus D}^{1}u(x)=0 \; m\mbox{-a.e. on } D\cap E_0.
$$ 
Since $m(E\setminus E_0)=0$ the above equality holds $m$-a.e. on $D$, and by the resolvent equation it leads to 
\begin{equation*}
    U_\alpha^Df=0 \; m\mbox{-a.e. on } D \mbox{ for each } \alpha >0. 
\end{equation*}
On the other hand, we have that $\mathop{\lim}\limits_{\alpha \to \infty}\alpha U_\alpha^D g=g$ pointwise on $D$ for every $g$ bounded and continuous on $G$.
Also, it is easy to see that $|U_\alpha^{D}|_{L^{p}(m|_D)}\leq |U_\alpha|_{L^{p}(m)}$ for every $\alpha>0$, and by a density argument it follows that $(U_\alpha^D)_{\alpha >0}$ is strongly continuous on $L^{p}(m|_D)$.

The result now follows since $f=\lim\limits_{\alpha \to \infty} \alpha U_\alpha^D f = 0$, the convergence being in $L^{p}(m|_D)$.
\end{proof}

\begin{proof}[\textit{\textbf{Proof of \Cref{coro 2.6}}}]
$(i)$ Let  $x \in E$ and set $\nu_x:=\delta_x \circ U_{\alpha_0}$.

By \Cref{prop 2.10}, we have that for any $\alpha>\alpha_0$, $\mathcal{U}_\alpha$ extends to strongly continuous resolvent on $L^1(E,\nu)$, whose generator is denoted by $({\sf L}_1^{\nu_x}, D({\sf L}_1^{\nu_x}))$, as in \eqref{2.2}.
By dominated convergence, we have that $\mathop{\lim}\limits_{t\to 0}\frac{P_tu-u}{t}=Lu$ in $L^{1}(E,\nu_x)$ for all $u\in \mathcal{C}$, hence $\mathcal{C}\subset D({\sf L}_1^{\nu_x}))$. 
Then property $(\rm{Loc}(G))$ holds for $G$ with respect to $\nu_x$. 
This is true for each $x\in E$, hence we can apply \Cref{thm 2.2} to deduce that $X$ is a diffusion in $G$ $\mathcal{U}$-a.e.
Then by \Cref{prop 2.4} (ii) applied for $X^G$ we deduce that $X^G$ is a diffusion on $G$, so the result follows by applying \Cref{prop 2.7}.

To prove (ii), let $u\in D_\alpha (L)\cap C_b(E)$ and notice that by Theorem \ref{thm 2.2}, (ii), applied for $\nu_x$ and each $x\in E$, we get that 
$$
Lu=0 \mbox{ on } \mathop{\wideparen{[u=0]}}\limits^\circ \quad \mathcal{U}\mbox{-a.e.}
$$
If $|Lu|$ is finely lower semi-continuous so that $[|Lu|>0]\cap \mathop{\wideparen{[u=0]}}\limits^\circ$ is finely open, hence empty according to Remark \ref{rem 2.2}, (i).
\end{proof}

\begin{proof}[\textit{\textbf{Proof of \Cref{coro:nonh}}}]
Let  $\overline{x} \in E\times[0,\infty)$ and set $\nu_{\overline{x}}:=\delta_{\overline{x}} \circ \overline{U}_{\alpha_0}$.
By \Cref{prop 2.10} we have that $\overline{\mathcal{U}}_{\alpha_0}$ extends to strongly continuous resolvent on $L^1(E\times[0,\infty),\nu_{\overline{x}})$, whose generator is denoted by $({\sf L}_1^{\nu_{\overline{x}}}, D({\sf L}_1^{\nu_{\overline{x}}}))$, as in \eqref{2.2}.
 Now one can easily show that $(\Lambda, D_{\alpha_0}(\Lambda))\subset ({\sf L}_1^{\nu_{\overline{x}}}, D({\sf L}_1^{\nu_{\overline{x}}}))$, and from now on the proof continues as for \Cref{coro 2.6}, (i).
\end{proof}

\begin{proof}[\textit{\textbf{Proof of \Cref{prop:Dnatural}}}]
Since the second part of the proposition can be easily deduced from the first part, let us prove only the first statement which regards the existence of the common nest.
To this end, we use again \Cref{lem:finer_top} from Appendix to assert that there exists a Ray topology $\mathcal{T}$ which is finer than both $\tau$ and $\tau'$, and furthermore, by Theorem 1.5 from \cite{BeBo05}, there exists an $m$-nest of $\mathcal{T}$-compact sets $(K_n)_n$. 
In particular, $(F_n)_n$ is an $m$-nest of closed sets with respect to both $\tau$ and $\tau'$.
Also, since $F_n$ is $\mathcal{T}$-compact, $\tau|_{F_n}=\mathcal{T}|_{F_n}=\tau'|_{F_n}$ for all $n\geq 1$.
\end{proof}

%%%%%%%%%%%%

\subsection{Proofs for \Cref{s:applications}}
\label{pf:3}

\begin{proof}[\textit{\textbf{Proof of \Cref{coro 5.1}}}]
First of all, note that by \Cref{prop 2.11} we have that $({\sf L}_p^{\overline{m}}, D({\sf L}_p^{\overline{m}}))$ introduced by \eqref{2.2} from Section 2 is a closed extension of $({\sf L},\mathcal{D})$. 

Let us show that $\rm{Loc}_m(G)$ holds for $G$:
Because $(e_n)_{n\geq 1}$ is total in $H$, following Remark \ref{rem 3.2}, (iv) we can construct $(f_n)_{n\geq 1}\subset \mathcal{D}$ such that $f_n\geq 0,n\geq 1$ and for each $(s,x) \neq (t,y) \in I\times M$, there exists $n\geq 1$ with $f_n(s,x)<f_n(t,y)$; in particular, condition (i) from \Cref{defi:locG} is satisfied with $F_n=I\times M, n\geq 1$. 

Let $(\varphi_k)_{k\geq 1} \subset C_c^\infty(\mathbb{R})$ such that  $0\leq \varphi_k(x) \mathop{\nearrow}\limits_{k}x^+, x\in \mathbb{R}$.
Then $\varphi_k(f_n-\varepsilon) \in \mathcal{D}$ for all $n,k\geq 1,\varepsilon\in \mathbb{R}_+$, hence condition (ii) from \Cref{defi:locG} is also verified.

Finally, since ${\sf L}f=0$ on $\mathop{\wideparen{[f=0]}}\limits^\circ \cap G$ for all $f\in \mathcal{D}$,  condition (iii) from \Cref{defi:locG} is satisfied, hence $\rm{Loc}_m(G)$ holds for $G$ and the statement follows by \Cref{thm 2.2}.
\end{proof}

\begin{proof}[\textit{\textbf{Proof of \Cref{coro:4.2}}}]
Let $\overline{x}\in I\times M$, and recall that by \Cref{prop 2.10},  $\overline{\mathcal{U}}_{\alpha_0}$ is strongly continuous on $L^1(\delta_{\overline{x}}\circ \overline{U}_{\alpha_0})$. 
Since by hypothesis $ \mathcal{D} \cup {\sf L}(\mathcal{D}) \subset L^1(\delta_{\overline{x}}\circ \overline{U}_{\alpha_0})$, we can apply \Cref{coro 5.1} for $\overline{m}=\delta_{\overline{x}}\circ \overline{U}_{\alpha_0}$ and $p=1$ to deduce that $X$ is a diffusion in $G$ $\delta_{\overline{x}}\circ\overline{U}_{\alpha_0}$-a.e., hence $\overline{\mathcal{U}}$-a.e. because $\overline{x}$ was arbitrarily chosen. 
Then by \Cref{prop 2.4}, (ii) it follows that the killed process $X^G$ is a diffusion on $G$, and by \Cref{prop 2.7} we deduce that $X$ is a diffusion in $G$.
\end{proof}

\begin{proof}[\textit{\textbf{Proof of \Cref{lem:bbound}}}]
If $F\in D_0(L)$ so that $F(\mu)=\psi(\mu(f_1),\cdots,\mu(f_n)), \mu \in M(E)$, and if we set $\overline{f}(x):=(f_1(x),\cdots, f_n(x)), x\in E$, then one can easily get from \eqref{eq:diff} that
    \begin{align*}
    F'(\mu;x)&=\left\langle D\psi(\mu(f_1),\cdots,\mu(f_n)),\overline{f}(x)\right\rangle,\\
    F''(\mu;x)&=\left\langle D^2\psi(\mu(f_1),\cdots,\mu(f_n))\overline{f}(x),\overline{f}(x)\right\rangle.
\end{align*}
Now, \eqref{eq:bbound1} follows by employing the definition of $LF(\mu)$ and the fact that $b,c$ and the kernel $n$ are bounded, whilst \eqref{eq:bbound2} follows from \eqref{eq:bbound1} and \cite[Proposition 2.5]{Fi88}.
\end{proof}

\begin{proof}[\textit{\textbf{Proof of \Cref{cor5.5}}}]
$i)\Longrightarrow ii).$ 
Let $z\in E$, $\mu \in G$ and $\varepsilon>0$ such that $\mu+\varepsilon \delta_z\in G$; note that such $\varepsilon>0$ exists because $G$ is open and $\lim\limits_{\varepsilon\to 0}\mu+\varepsilon\delta_z=\mu$ weakly in $M(E)$.

Further, let $0\leq\psi\in C_c^\infty(\mathbb{R})$ such that $[0,\varepsilon]\subset \mathbb{R}\setminus supp(\psi)$ and consider
\begin{align*}
 \psi_\mu(x)&:=\psi(x-\mu(1)), \quad x\in \mathbb{R}, \\
 F_\mu(\nu)&:=\psi_\mu(\nu(1)), \quad \nu\in M(E).
\end{align*}
Notice that $F_\mu \in D_{00}(L)$ and $\mu+\varepsilon\delta_z\in \mathop{\wideparen{[F_\mu=0]}}\limits^\circ\cap G$.

Now, on the one hand, since $X$ is a diffusion in $G$, by \Cref{lem:bbound} and \Cref{coro 2.6}, (ii) we get that $LF_\mu(\mu+\varepsilon\delta_z)=0$.
On the other hand, since $F_\mu'(\mu+\varepsilon\delta_z;x)=\psi'(\varepsilon)=0=\psi''(\varepsilon)=F_\mu''(\mu+\varepsilon\delta_z;x)$ for all $x\in E$, one can easily see that
\begin{align*}
0&=LF_\mu(\mu+\varepsilon\delta_z)=\int_E\int_0^\infty F_\mu(\mu+\varepsilon\delta_z+u\delta_x) n(x,du)(\mu+\varepsilon\delta_z)(dx)\\
&\geq \varepsilon \int_0^\infty\psi(\varepsilon+u)n(z,du).
\end{align*}
Now, varying $\varepsilon$ and $\psi$, by a monotone class argument we deduce that $n(z,\cdot)\equiv 0$, and since $z\in E$ was chosen arbitrarily, the implication is proved.

\medskip
\noindent{$ii)\Longrightarrow iii).$}
Taking into account \Cref{lem:bbound}, this implication is easily deduced from \Cref{coro 2.6}, (i).

The implication $ii)\Longrightarrow iii)$ is trivial, so the result is completely proved.
\end{proof}

\begin{proof}[\textit{\textbf{Proof of \Cref{coro:ibranch}}}]
If $(\widetilde{Q}_t)_{t\geq 0}$ denotes the transition function of $X$, then similarly to the proof of \Cref{lem:bbound}, we get that $(\widetilde{L}, \widetilde{D}_{00}(\widetilde{L}))\subset (\widetilde{L},D_\alpha(\widetilde{L}))$ for $\alpha>\alpha_0$, where the former generator is the one given by \eqref{eq:dgen}-\eqref{eq:gen} with $V_u(\mu):=c\mu(1), \mu\in M(E)$, where $c>0$ is a generic constant.

Next, it is easy to see that condition $\rm Loc(M(E))$ given in \Cref{defi:pLocG} is satisfied by choosing $(f_n)_n\subset \widetilde{D}_{00}(\widetilde{L})$ with the required separation property, so \Cref{coro 2.6} can be immediately employed to get the result.
\end{proof}

\begin{proof}[\textit{\textbf{Proof of \Cref{lem 3.1}}}]
Let $D(L_0)\ni u:=U_\alpha f \in D(\Lambda, \mathcal{H})$ for some $f\in \mathcal{H}$.
The statement follows if we show that
\begin{equation} \label{eq 3.2}
G_\alpha f=U_\alpha f + G_\alpha\Lambda U_\alpha f.
\end{equation}
Indeed, if \eqref{eq 3.2} holds then $u= G_\alpha (f-\Lambda U_\alpha f) \in D(L^{\Lambda})$ and
\begin{equation*}
L^{\Lambda} u = \alpha G_\alpha (f-\Lambda U_\alpha f) -(f-\Lambda U_\alpha f) = \alpha U_\alpha f - f + \Lambda u = L_0u +\Lambda u.
\end{equation*}

To prove \eqref{eq 3.2}, note that by \eqref{eq 3.1}, the rhs of the equality belongs to $\mathcal{F}$, and for all $v\in \mathcal{V}$
\begin{align*}
\mathcal{E}_\alpha (U_\alpha f + G_\alpha\Lambda U_\alpha f, v)&=\mathcal{E}_\alpha^{0}(U_\alpha f, v) - \langle \Lambda U_\alpha f, v \rangle + \mathcal{E}_\alpha (G_\alpha \Lambda U_\alpha f, v)\\
&=\langle f, v \rangle - \langle \Lambda U_\alpha f, v \rangle + \langle \Lambda U_\alpha f, v \rangle \\
&=\langle f, v \rangle.
\end{align*}
The claim follows by the uniqueness of $W_\alpha (= G_\alpha \mbox{ on }\mathcal{H})$ satisfying \eqref{eq 3.1}.
\end{proof}

\begin{proof}[\textit{\textbf{Proof of \Cref{coro:nonqr}}}]
First of all, notice that by \Cref{lem 2.11} and \Cref{prop 2.10} we may assume without loss of generality that $p=1$, $m$ is finite and $(P_t)_{t>0}$ is a semigroup of (quasi-)contractions.
In particular, $\mathcal{C}\subset D({\sf L}_1^m)$.

Using that the resolvent $(V_\alpha)_{\alpha>0}$ is strongly continuous on $L^1$, as in the proof of Theorem 2.2 from \cite{BeBoRo06} (see also Lemma 2.1), we can construct a set $F\in \mathcal{B}$ and a resolvent of sub-Markovian kernels $\cw=(W_\alpha)_{\alpha>0}$ on $F$ such that
\begin{enumerate}
    \item[-] $m(E\setminus F)=0$
    \item[-] $W_\alpha = V_\alpha,\alpha >0$,  as operators on $L^p(E,m)$
    \item[-] The set of $\ce(\cw_\beta)$ all $\cw_\beta$-excessive functions is min-stable, contains the positive constant functions and generates $\cb|_{F}$, and for each $n\geq 1$, we have
    $f_n|_{F}=W_\beta g_n^+|_{F}-W_\beta g_n^-|_{F}$.
\end{enumerate}
Let $\calr$ be a Ray cone associated with $\cw_\beta$ such that $k\wedge W_\beta g_n^+|_{F}\in \calr$ for all
$k,n \geq 1$, and consider the saturation $F_1$ 
of $F$ with respect to $\cw_\beta$ (see Appendix for details) and the
resolvent of kernels $\cw^1=(W^1_\al)_{\al >0}$ on $(F_1,\cb_1)$, 
whose restriction to $F$ is $\cw|_F$ and $W^1_\al (1_{F_1\setminus
F})=0$.  
We endow $F_1$ with the (Ray) topology $\ct$ induced by $\widetilde{\calr}:= \{ \widetilde{v}: v\in \calr \}$, where $\widetilde{v}$ is the unique $\mathcal{W}^1_\beta$-excessive extension of $v\in \mathcal{R}$ from $F$ to $F_1$.
It is a Lusin
topology on $F_1$ such that $\cb_1$ is the $\sigma$-algebra of all
Borel sets on $E_1$ and $\cw^1$ is the resolvent of a right
process $X^1$ with state space $F_1$, we have
$W^1_\al =V_\al$ for all
$\al >0$, regarded as an equality of operators on $L^p(F_1, {m}_1),$
where $m_1$ is the measure on $(F_1, \cb_1)$ extending $m|_F$ with zero on $F_1\setminus F$.

Because condition ${\bf (D_m^p)}$ is in force,
it follows by Theorem 3.3 from \cite{BeBoRo06} and Proposition 5.1 from \cite{BeRo11a} 
that $X$ has c\`adl\`ag trajectories, possibly after a trivial modification on an $m_1$-innesential set.

Now we show that process $X^1$ is a diffusion on $F_1$ $m_1$-q.e., by applying  Theorem \ref{thm 2.2}.
To this end, let $\widetilde{f_{n}}:=W_\beta^1g_n^+-W_\beta^1g_n^-$ be the extension of $f_n|_{F}$, $n, \geq 1$, from $F$ to $F_1$ by $\mathcal{T}$-continuity. 
We claim that condition ${\rm Loc_m(F_1)}$ is satisfied on $F_1$, with the augmented sequence $\{\widetilde{f_{n}},|\widetilde{f_{n}}|_\infty-\widetilde{f_{n}}\}_{n\geq 1}$ instead of  $({f_n})_{n\geq 1}$.
To this end, we first claim that $\{\widetilde{f_{n}},|\widetilde{f_{n}}|_\infty-\widetilde{f_{n}}\}_{n\geq 1}$ separates the points of $F_1$ in the sense of \Cref{defi:locG}, (i):
Let $\zeta, \eta \in E_1$ and assume that $\widetilde{f_{n}}(\zeta)=\widetilde{f_{n}}(\eta)$ for all $n\geq 1$.
From the above considerations we have
$ \zeta(g_{n})= L_\beta (\zeta, W_\beta g_{n})=\widetilde{f_{n}}(\zeta)=  \widetilde{f_{n}}(\eta)= \eta(g_{n})$ for all $n\geq 1$ and we conclude that $\zeta=\eta$,
since we assumed that the sequence $(g_n)_{n\geq 1}$ separates the finite measures on $E$.
Therefore, $(\widetilde{f_n})_n$ separates the points of $F_1$.
But if $\widetilde{f_n}(\zeta)<\widetilde{f_n}(\eta)$, then $(|\widetilde{f_{n}}|_\infty-\widetilde{f_{n}})(\zeta)>(|\widetilde{f_{n}}|_\infty-\widetilde{f_{n}})(\eta)$, so the claim is proved.

Next, condition \Cref{defi:locG}, (ii) is a consequence of $(\ref{5.1})$ from assumption $\rm\mathbf{\widetilde{Loc}_m}$ and the initial remark that $\mathcal{C}\subset D({\sf L}_1^m)$, whilst \Cref{defi:locG}, (iii) is clearly fulfilled.
Consequently, $X^1$ is a diffusion on $F_1$ $m_1$-q.e.

Let $E_o$ be the disjoint union of $F_1$ and $E\setminus F$, endow $E\setminus F$ 
with any Lusin topology having as Borel $\sigma$-algebra $\cb|_{E\setminus F}$, 
and consider the trivial extension $X$ of $X^1$ from $F_1$ to $E_o$; for details see, e.g., Subsection 3.2 from \cite{BeRo11a}. 
Then clearly $X$  has continuous paths $m_o$-q.e. 
and  the  transition function  of $X$,  
regarded as a family of operators on $L^p(E_o, m_o)$,  
coincide  with $(P_t)_{t\geq 0}$,  where $m_o$ is the measure on $(E_o, \cb_o)$ extending $m$ with zero on $E_o\setminus E=F_1\setminus F$.

\end{proof}

\section{Appendix: Basics on right processes with c\`adl\`ag trajectories}

\paragraph{Excessive functions, natural topologies and right processes.}
Here we follow mainly the terminology of \cite{BeBo04a}, and we refer to the classical works \cite{BlGe68}, \cite{Sh88} and the references therein. 

Let $(E, \mathcal{B})$ be a Lusin measurable space.
We denote by $(b)p\mathcal{B}$ the set of all numerical, (bounded) positive $\mathcal{B}$-measurable functions on $E$.
Throughout, by $\mathcal{U}=(U_\alpha)_{\alpha>0}$ we denote a resolvent family of (sub-)Markovian kernels on $(E, \mathcal{B})$.
If $q >0$, we set $\mathcal{U}_q:=(U_{q+\alpha})_{\alpha>0}$.

\begin{defi} \label{defi 4.1}
A $\mathcal{B}$-measurable function $v:E\rightarrow \overline{\mathbb{R}}_+$ is called {\rm excessive} (w.r.t. $\mathcal{U}$) if $\alpha U_\alpha v \leq v$ for all $\alpha >0$ and $\mathop{\sup}\limits_\alpha \alpha U_\alpha v =v $ point-wise; by $\mathcal{E(\mathcal{U})}$ we denote the convex cone of all excessive functions w.r.t. $\mathcal{U}$.
\end{defi}

Assume that $\mathcal{U}$ is the resolvent of a right Markov  process 
$X=(X(t), \mathcal{F}_t, \mathbb{P}^x)$ 
with state space $E$, a Lusin topological space. 
Then a non-negative real valued $\mathcal{B}$-measurable function $v$  is excessive (w.r.t. $\mathcal{U}$) if and only if $(v(X(t)))_{t\geqslant 0}$ is a right continuous $\mathcal{F}_t$ -supermartingale w.r.t. $\mathbb{P}^x$  for all $x\in E$; see e.g.
Proposition 1 from \cite{BeCi18}.

If a $\mathcal{B}$-measurable function $w: E \rightarrow \overline{\mathbb{R}}_+$ is merely $\mathcal{U}_q$-supermedian (i.e. $\alpha U_{q+\alpha} w \leq w$ for all $\alpha > 0$), then its $\mathcal{U}_q${\it -excessive regularization} $\widehat{w} \in \mathcal{E(U)}$ is defined as
$\widehat{w}:=\mathop{\sup}\limits_\alpha \alpha U_{q +\alpha}w$.

\noindent
{\bf (H)} Throughout this paragraph we assume that $\mathcal{E}(\mathcal{U}_q)$ is min-stable, contains the constant functions, and generates $\mathcal{B}$ for one (hence all) $q >0$.

\medskip
Recall that {\bf (H)} is necessary (yet not sufficient) for $\mathcal{U}$ to be associated to a right process, as defined below; a practical way to check this condition for a given resolvent of kernels is given in e.g. \cite{BeRo11a}, page 846, and it is similar to $\mathbf{(H_0)}$ from the beginning of Section 2.

\begin{defi} \label{defi 3.3}
\begin{enumerate}[(i)]
\item The {\rm fine topology} on $E$ (associated with $\mathcal{U}$) is the coarsest topology on $E$ 
such that every $\mathcal{U}_q$-excessive function is continuous for some (hence all) $q>0$.
\item A topology $\tau$ on $E$ is called {\rm natural} if it is a Lusin topology (i.e. $(E,\tau)$ is homeomorphic to a Borel subset of a compact metrizable space) which is coarser than the fine topology, and whose Borel $\sigma$-algebra is $\mathcal{B}$.
\end{enumerate}
\end{defi}
\begin{rem} \label{rem 4.3}
The necessity of considering natural topologies comes from the fact that, in general, the fine topology is neither metrizable, nor countably generated.
\end{rem}

There is a convenient class of natural topologies to work with (as we do in Section 2), especially when the aim is to construct a right process associated to $\mathcal{U}$ (see Definition \ref{defi 4.4}). These topologies are called Ray topologies, and are defined as follows.

\begin{defi} \label{defi 4.5}
\begin{enumerate}[(i)]
\item If $q >0$ then a {\rm Ray cone}
associated with $\mathcal{U}_q$ is a cone $\mathcal{R}$ of bounded $\mathcal{U}_q$-excessive functions which is separable in the supremum norm, min-stable, contains the constant function $1$, generates $\mathcal{B}$, and such that $U_\alpha((\mathcal{R}-\mathcal{R})_+) \subset \mathcal{R}$ for all $\alpha > 0$.
\item A {\rm Ray topology} on $E$ is a topology generated by a Ray cone.
\end{enumerate} 
\end{defi}

\begin{rem} \label{rem 4.6}
\begin{enumerate}[(i)]
\item Clearly, any Ray topology is a natural topology.
\item By e.g. \cite{BeRo11a}, Proposition 2.2, a Ray cone always exists (its existence is  in fact equivalent with the validity of {\bf (H)}) and may be constructed as follows: start with a countable subset $\mathcal{A}_0\subset p\mathcal{B}$ which separates the points of $E$, and define inductively 
\begin{align*}
&\mathcal{R}_0:=U_q (\mathcal{A}_0)\cup \mathbb{Q}_+\\
&\mathcal{R}_{n+1}:=\mathbb{Q}_+\cdot \mathcal{R}_n \cup (\mathop{\sum}\limits_f\mathcal{R}_n )\cup (\mathop{\bigwedge}\limits_f \mathcal{R}_n) \cup (\mathop{\bigcup}\limits_{\alpha \in \mathbb{Q}_+}U_\alpha(\mathcal{R}_n))\cup U_q ((\mathcal{R}_n-\mathcal{R}_n)_+),
\end{align*}
where by $\mathop{\sum}\limits_f \mathcal{R}_n$ resp. $\mathop{\bigwedge}\limits_f\mathcal{R}_n$ we denote the space of all finite sums (resp. infima) of elements from $\mathcal{R}_n$.
Then, a Ray cone $\mathcal{R}$ is obtained by taking the closure of $\bigcup\limits_n \mathcal{R}_n$ w.r.t. the supremum norm.
\end{enumerate}
\end{rem}
The set of all natural topologies has the following remarkable structure:
\begin{lem} \label{lem:finer_top}
For any two natural topologies $\tau$ and $\tau'$ there exists a Ray (hence natural) topology which is finer than both $\tau$ and $\tau'$.
\end{lem}
\begin{proof}
By Proposition 2.1 from \cite{BeBo05} if a natural topology is given then there exists a Ray topology which is finer than it. 
Therefore, we may assume that $\tau$ and $\tau'$ are Ray topologies induced by the Ray cones $\mathcal{R}$ and $\mathcal{R}'$ respectively and we may construct a Ray cone such that its closure in the supremum norm includes both $\mathcal{R}$ and $\mathcal{R}'$.
\end{proof}

Let now $X=(\Omega, \mathcal{F}, \mathcal{F}_t , X(t), \theta(t) , \mathbb{P}^x)$ be a normal Markov process with state space $E$, shift operators $\theta(t):\Omega\rightarrow \Omega, \; t\geq 0$, and lifetime $\zeta$. 
Let $\mathcal{U}$ be the resolvent of $X$, i.e. for all $f\in b\mathcal{B}$ and $\alpha >0$
$$
U_\alpha f(x)=\mathbb{E}^{x}\Big\{\int_0^{\infty}e^{-\alpha t}f(X(t))dt\Big\},\quad  x\in E.
$$
To each probability measure $\mu$ on $(E, \mathcal{B})$ we associate the probability 
$$\mathbb{P}^\mu (A):=\mathop{\int} \mathbb{P}^x(A)\; \mu(dx)$$
for all $A \in \mathcal{F}$, and we consider the following enlarged filtration
$$
\widetilde{\mathcal{F}}_t:= \bigcap\limits_\mu \mathcal{F}_t^\mu, \; \; \widetilde{\mathcal{F}}:= \bigcap\limits_\mu \mathcal{F}^\mu,
$$
where $\mathcal{F}^\mu$ is the completion of $\mathcal{F}$ under $\mathbb{P}^\mu$, and $\mathcal{F}_t^\mu$ is the completion of $\mathcal{F}_t$ in $\mathcal{F}^\mu$ w.r.t. $\mathbb{P}^\mu$; in particular, $(x,A)\mapsto\mathbb{P}^{x}(A)$ is assumed to be a kernel from $(E,\mathcal{B}^{u})$ to $(\Omega, \mathcal{F})$, where $\mathcal{B}^{u}$ denotes the $\sigma$-algebra of all universally measurable subsets of $E$.

\begin{defi} \label{defi 4.4}
The Markov process $X$ is called 
{\rm right (Markov) process} if the following additional hypotheses are satisfied:
\begin{enumerate}[(i)]
\item The filtration $(\mathcal{F}_t)_{t\geq 0}$ is right continuous and $\mathcal{F}_t=\widetilde{\mathcal{F}}_t, t\geq 0$.
\item For one (hence all) $q >0$ and for each $f \in \mathcal{E}(\mathcal{U}_q)$ the process $f(X)$ has right continuous paths $\mathbb{P}^{x}$-a.s. for all $x\in E$.
\item There exists a natural topology on $E$ with respect to which the paths of $X$ are $\mathbb{P}^{x}$-a.s. right continuous for all $x\in E$.
\end{enumerate}
We would like to make the following convention: Whenever the space $E$ is given along with a Lusin topology $\tau$ and there is no risk of confusion, by saying that $X$ is a right process we implicitly assume that $X$ has $\mathbb{P}^x$-a.s. $\tau$-right continuous paths for all $x\in E$. 
\end{defi}

According to \cite{BlGe68}, Chapter II, Theorem 4.8, or \cite{Sh88}, Proposition 10.8 and Exercise 10.18, Definition \ref{defi 4.4} leads to a key probabilistic understanding of the fine topology, namely:

\begin{thm} \label{thm 4.6} If $X$ is a right process, then an universally $\mathcal{B}$-measurable function $f$ is finely continuous if and only if $(f(X(t)))_{t\geq 0}$ has $\mathbb{P}^{x}$-a.s. right continuous paths for all $x\in E$.
In particular,  $X$ has a.s. right continuous paths w.r.t. any natural topology on $E$
\end{thm}

\begin{defi}
If $u\in \mathcal{E}(\mathcal{U}_q)$ and $A \in \mathcal{B}$, then the $q$-order reduced function of $u$ on $A$ is given by
$$
R_q^A u = \inf \{ v \in \mathcal{E}(\mathcal{U}_q): \, v\geqslant u \mbox{ on } A \}.
$$
$R_q^A u$ is merely supermedian w.r.t.  $\mathcal{U}_q$, and we denote by $B_q^A u=\widehat{R_q^A u}$ its excessive regularization, called the {\rm balayage} of $u$ on $A$.
\end{defi}

The following fundamental identification due to G.A. Hunt holds (see e.g. \cite{DeMe78}):
\begin{thm} \label{thm 4.13}
If $X$ is a right process and $q>0$,  then for all $u\in \mathcal{E}(\mathcal{U}_q)$ and $A\in\mathcal{B}$
$$
B_q^A u=\mathbb{E}^x\{e^{-q T_A} u (X(T_A))\},
$$
where $T_A:=\inf \{ t>0 :  X(t)\in A \}$.
\end{thm}

It is well known that $T_A$ is a stopping time and that $B^A_q u$ is universally measurable for all $A\in \mathcal{B}$ and $u\in b\mathcal{B}$;
see \cite{BlGe68} or \cite{Sh88}.

\paragraph{Notions of "small" sets.}

\begin{defi} 
Let $m$ be a $\sigma$-finite measure on $E$.
\begin{enumerate}[(i)]
\item A set $A\in \mathcal{B}$ is called
\begin{enumerate}
\item[-] \say{ $\mathcal{U}$-negligible} if $U_q (1_A)\equiv 0$ for one (hence all) $q> 0$.
\item[-] \say{ polar}  if $\mathbb{P}^{x}(T_A<\infty)=0$ for $x\in E$.
\item[-] \say{ $m$-polar}  if $\mathbb{P}^{x}(T_A<\infty)=0$ for all $x\in E$ $m$-a.e.
\item[-] \say{ $m$-inessential}  provided that it is $m$-negligible and $\mathbb{P}^{x}(T_A<\infty)=0$ for all $x\in E\setminus A$.
\end{enumerate}
\item A property is said to hold $m${\it -quasi-everywhere} (resp. $\mathcal{U}$-a.e.), if there exists an $m$-inessential (resp. a $\mathcal{U}$-negligible) set $N$ s.t. the property holds for all $x\in E\setminus N$; on short, we write $m${\it -q.e.} instead of $m$-quasi-everywhere.
\end{enumerate}
\end{defi}

\begin{rem} \label{rem 2.2} 
For the reader convenience, let us recall several potential theoretic facts. 
\begin{enumerate}[(i)]
\item If $A \in \mathcal{B}$ is finely open and $\mathcal{U}$-negligible, then $A=\emptyset$.
\item If $m$ is a $\sigma$-finite measure on $E$ s.t. 
$m(A)=0$ implies $U_1(1_A)=0$ $m$-a.e. for all $A\in \mathcal{B}$, then any finely open and $m$-negligible set $A\in \mathcal{B}$ is $m$-polar. 
\item Any $m$-inessential set is $m$-polar; conversely, it is known that any set which is $m$-polar and $m$-negligible is the subset of an $m$-inessential set.
\item If $u\in b\mathcal{B}$ and $v$ is $\mathcal{B}$-measurable s.t. $v=u$ q.e., then $B^A_{1}v$ is well defined and equal to $B^A_{1}u$ $m$-a.e.
\end{enumerate}
\end{rem}

\paragraph{C\`adl\`ag paths in different topologies.}
As far as we know, the stability of the right continuity of the paths under the change of the (natural) topology ensured by \Cref{thm 4.6} can not be simply extended for left limits, without further conditions. 
To present such a condition, we adopt an $L^p$-framework, so let $m$ be a $\sigma$-finite measure on $E$ such that the resolvent $\mathcal{U}$ of $X$ is strongly continuous on $L^{p}(m)$ for some $1\leq p<\infty$.

Let us recall that an $q$-excessive function $s$ is called {\it regular} if for every sequence of $q$-excessive functions $u_n\mathop{\nearrow} \limits_n u$ it holds that $R_q (u-u_n)\mathop{\searrow} \limits_n 0$, where $R_q$ is the {\it reduction operator} of level $q \geq 0$. 
A $q$-excessive function $s$ is called $m$-regular if it has an $m$-version which is regular; see \cite{BeBo00} and \cite{BeBo05} for more details.

Consider the following {\it domination hypothesis}:

\vspace{0.2 cm}
\noindent
$\mathbf{(D_{m}).}$ There exists $0<f_0\in b\mathcal{B}(E)$ such that for some $q >0$ every $q$-excessive function $v$ dominated by $U_q f_0$ is $m$-regular.

The role of condition $\mathbf{(D_{m})}$ is expressed by the following fact which is a consequence of the results from \cite{BeBo05}. 

\begin{prop} \label{prop2.6}
If $X$ is a right process and $\mathbf{(D_{m})}$ holds then $X$ has $\mathbb{P}^x$-a.s. c\`adl\`ag trajectories in $E$ on $[0, \zeta)$ for all $x\in E$ $m$-q.e., with respect to all natural topologies.
\end{prop}

\begin{proof}
Let $\tau$ be an arbitrary natural topology on $E$. 
By \Cref{lem:finer_top} from Appendix, we may replace $\tau$ with a finer Ray topology on $E$. 
Now, with condition $\mathbf{(D_{m})}$ in force, Theorems 1.5 and 1.3 from \cite{BeBo05} entail that $X$ has c\`adl\`ag trajectories in $E$ on $[0, \zeta)$ with respect to $\tau$, $\mathbb{P}^m$-a.e.
Now, by the proof of Proposition 5.1 from \cite{BeRo11a} we have that the function 
\begin{equation*}
E\ni x\mapsto \mathbb{P}^x(\omega \in \Omega : t\mapsto X_t(\omega) \mbox{ is not c\`adl\`ag with respect to } \tau) 
\end{equation*}
is excessive, hence $[v>0]$ is finely open.
Also, the above discussion leads to $m([v>0])=0$, hence by \Cref{rem 2.2} we get that $[v>0]$ is contained in an $m$-inessential set.
In other words, $X$ has c\`adl\`ag trajectories with respect to $\tau$, $\mathbb{P}^x$-a.s. for all $x\in E$ $m$-q.e. 
\end{proof}

\paragraph{Existence of a right process with a given resolvent $\mathcal{U}$.}
Without further conditions, the assumption {\bf (H)} from the beginning of this section, although necessary, is not sufficient to ensure the existence of a right process associated with $\mathcal{U}$, but there is always a larger space on which such a process exists, and let us briefly recall its construction.

We denote by $Exc(\mathcal{U}_q)$ the set of all $\mathcal{U}_q$-excessive measures:
$\xi \in Exc(\mathcal{U}_q)$ if and only if $\xi$ is a $\sigma$-finite measure on $E$ and $\xi \circ \alpha U_{q+\alpha} \leq \xi$ for all $\alpha >0$.

\begin{defi} Let $q >0$.
\begin{enumerate}[(i)]
\item The {\it energy functional} associated with $\mathcal{U}_q$ is $L^{q}: Exc(\mathcal{U}_q)\times \mathcal{E}(\mathcal{U}_q) \rightarrow \overline{\mathbb{R}}_+$ given by
$$
L^{q}(\xi,v):=\sup\{\mu(v) \; : \; \mu \mbox{ is a } \sigma\mbox{- finite measure, } \mu \circ U_q \leq \xi\}
$$
\item The {\rm saturation} of $E$ (with respect to $\mathcal{U}_q$) is the set $E_1$ of all extreme points of the set $\{\xi \in Exc(\mathcal{U}_q)\; : \; L^{q}(\xi,1)=1\}$.
\end{enumerate}
\end{defi}

The map $E\ni x \mapsto \delta_x \circ U_q \in Exc(\mathcal{U}_q)$ is an embedding of $E$ into $E_1$ and 
every $\mathcal{U}_q$-excessive function $v$ has an extension $\widetilde{v}$ to $E_1$, defined as $\widetilde{v}(\xi):=L^{q}(\xi,v)$.
The set $E_1$ is endowed with the $\sigma$-algebra $\mathcal{B}_1$ generated by the family $\{\widetilde{v}: \; v\in \mathcal{E}(\mathcal{U}_q)\}$.
In addition, as in \cite{BeBoRo06}, sections 1.1 and 1.2, there exists a unique resolvent of kernels $\mathcal{U}^{1}=(U^{1}_\alpha)_{\alpha>0}$ on $(E_1, \mathcal{B}_1)$ 
which is an extension of $\mathcal{U}$ in the sense that $U^1_\alpha(1_{E_1\setminus E})\equiv 0$ and $(U^1_\alpha f)|_E=U_\alpha(f|_E)$ for all $f\in b\mathcal{B}_1,\alpha>0$, and it satisfies the assumption {\bf(H)} from the beginning of this section; more precisely, it is given by
\begin{equation} \label{eq 4.1}
U^{1}_\alpha f(\xi)=L^{q}(\xi, U_\alpha (f|_E)) 
\mbox{ for all } f \in bp\mathcal{B}_1, \xi\in E_1, \alpha >0.
\end{equation}
Notice that $(E_1,\mathcal{B}_1)$ is a Lusin measurable space, the map $x \mapsto\delta_x \circ U_q$ identifies $E$ with a subset of $E_1$, $E\in \mathcal{B}_1$ and $\mathcal{B}=\mathcal{B}_1|_E$.

We end this section with the following key result, according to (2.3)  from \cite{BeRo11a},  sections 1.7 and 1.8 in \cite{BeBo04a},  Theorem 1.3 from \cite{BeBoRo06}, and section  3 in \cite{BeBoRo06a}:

\begin{thm} \label{thm 4.15}
Suppose that assumption {\bf(H)} from the beginning of this section is satisfied.
Then there exists a right process on the saturation $(E_1,\mathcal{B}_1)$, associated with $\mathcal{U}^{1}$. Moreover, the following assertions are equivalent:
\begin{enumerate}[(i)]
\item There exists a right process on $E$ associated with $\mathcal{U}$.
\item The set $E_1\setminus E$ is polar (w.r.t. $\mathcal{U}_1$).
\end{enumerate}
\end{thm}

\paragraph{Existence of a right process with c\`adl\`ag trajectories associated with the resolvent $\mathcal{U}$.}
We end this section by recalling a general result that guarantees the existence of a right process with c\`adl\`ag trajectories associated to a given sub-Markovian resolvent of kernels $\mathcal{U}=(U_\alpha)_{\alpha>0}$ on a general Lusin measurable space $(E,\mathcal{B})$.

\vspace{0.2 cm}
\noindent
{\bf (Hypothesis).} Assume that $\mathcal{C}\subset b\mathcal{B}$ such that
\begin{enumerate}
\item[(H1)] $\mathcal{C}$ is a vector lattice, $1\in \mathcal{C}$.
\item[(H2)] There exists a countable subset of $\mathcal{C}_+$ which separates the points of $E$. %such that for all $f\in \mathcal{C}_0$
%$$
%\lim\limits_{\alpha \to \infty} \alpha U_\alpha f=f \mbox{ uniformly on each } [v\leq n], n\geq 1. {\color{red} not \; needed, it follows...}
%$$
\item[(H3)] $U_\alpha (\mathcal{C})\subset \mathcal{C}$ for all $\alpha>0$.
\item[(H4)] $\lim\limits_{\alpha \to \infty} \alpha U_\alpha f=f$ pointwise on $E$.
\item[(H5)] There exists a $\mathcal{B}$-measurable function $v\in \mathcal{E}(\mathcal{U}_q)$ for some $q \geqslant 0$ such that $v<\infty$ 
%(resp. $m-a.e.$) 
and $[v\leqslant n],n\geqslant 1$, is $\tau(\mathcal{C})$-compact, 
$n\geqslant 1$,  is $\tau(\mathcal{C})$-compact; here $\tau(\mathcal{C})$ denotes the topology on $E$ generated by $\mathcal{C}$.
\end{enumerate} 

\begin{thm}\label{thm:cadlag}
If (H1)-(H5) hold, then there exists a c\`adl\`ag right process $X$ on $E$, endowed with the topology $\tau(\mathcal{C})$,  with resolvent $\mathcal{U}$. 
\end{thm}

\begin{proof}
By Proposition 2.2 from \cite{BeRo11a} there exists a Ray cone $\mathcal{R}$ such that the Ray topology generated by $\mathcal{R}$ is smaller than $\tau(\mathcal{C})$.
Let $K_n=[v\leqslant n]$, $n\geqslant 1.$ 
It is an increasing sequence of $\tau(\mathcal{C})$-compact
and by Theorem \ref{thm 4.13} if follows that
$\lim_n T_n= + \infty$ $\mathbb{P}^x$-a.s. 
for every $x\in E$, where $T_n:= \inf \{ t>0 : X(t) \in E\setminus K_n \}.$ 
Notice that the Ray topology and $\tau(\mathcal{C})$ coincide on each compact set $K_n$.
Applying Lemma 3.5 from \cite{BeBoRo06a} and Theorem \ref{thm 4.15} it follows that there exists a right process $X$ with state space $E$ and $\mathcal{U}$ as associated resolvent.
By Theorem 1.3 from \cite{BeBo05} we conclude now that $X$ has c\`adl\`ag trajectories in the topology $\tau(\mathcal{C})$.
\end{proof}

\begin{rem}  Results related to the domination hypothesis (${\bf D_m}$) and to hypothesis (H5) may be found in \cite{BeTr11} and \cite{BeRo11b}, in terms of the tightness property for the associated capacities;
see also \cite{LyRo92}.
\end{rem} 

\vspace{2mm}

%A converse is also true.\\[5mm]

\noindent \textbf{Acknowledgements.} 
Funded by the Deutsche Forschungsgemeinschaft (DFG, German Research Foundation)-SFB 1283/2 2021-317210226. 
This work was supported by grants of the Ministry of Research, Innovation and Digitization, CNCS - UEFISCDI,
project number PN-III-P4-PCE-2021-0921, within PNCDI III for the first author, and
project number PN-III-P1-1.1-PD-2019-0780, within PNCDI~III for the second author.

\end{document}